\documentclass{article}
\usepackage{amsmath}
\usepackage{amssymb}

\input xypic
\input xy
\xyoption{all}
\newtheorem{theorem}{Theorem}

\newtheorem{definition}[theorem]{Definition}

\newtheorem{proposition}[theorem]{Proposition}
\newtheorem{remark}[theorem]{Remark}

\newenvironment{proof}[1][Proof]{\noindent\textbf{#1.} }{\ \rule{0.5em}{0.5em}}
\input{tcilatex}
\begin{document}

\author{ U. E. Arslan and G. Onarl\i\ }
\title{\text{Categorical Results in the Theory of} \\{ Two-Crossed Modules of
Commutative Algebras }}
\date{}
\maketitle

\begin{abstract}
In this paper we explore some categorical results of $2$-crossed module of commutative algebras extending work of Porter in \cite{[p1]}. We also show that the forgetful functor from the category of $2$-crossed modules to the category of $k$-algebras, taking $\{L,M,P,\partial _{2},\partial _{1}\}$ to the base algebra $P$, is fibred
and cofibred considering the pullback (coinduced) and induced $2$-crossed modules constructions, respectively. Also we consider free $2$-crossed modules as an application of induced $2$-crossed modules.
\end{abstract}

\section*{Introduction}
Crossed modules of groups were initially defined by Whitehead \cite%
{[w1],[w2]} as models for (homotopy) $2$-types. The commutative algebra case of crossed modules
is contained in the paper of Lichtenbaum and Schlessinger  \cite{L-S} and also the work of
Gerstenhaber \cite{G} under different names. Some categorical results and Koszul complex link are also
given by Porter \cite{[p1],[p2]}. Conduch\'{e}, \cite{cond}, in 1984 described the notion of $2$-crossed modules as a model for $3$-types. The commutative algebra version of $2$-crossed modules has been
defined by Grandje\'{a}n and Vale \cite{GV}. Arvasi and Porter
\cite{ap97},\cite{a} have important studies related with that
construction.

The purpose of this paper is to investigate some categorical theory of $2$-crossed modules.
It is considered the results as easy to prove but nonetheless giving some functorial relations between the category \textsf{X$_{2} $Mod} of $2$-crossed modules of commutative algebras and other categories with some adjoint pairs of functors would seem to be important to find out the question as to whether or not \textsf{X$_{2} $Mod} is a fibred category.

Here we will give the construction of the pullback and induced $2$%
-crossed modules of commutative algebras extending results of Shammu in \cite%
{[s]} and Porter in \cite{[p1]} for crossed modules of algebras. The construction of
pullback and induced $2$-crossed modules will give us a pair of
adjoint functors $\left( \phi ^{\ast },\phi _{\ast }\right) $ where
\begin{equation*}
\phi ^{\ast }\text{ }:\text{\textsf{X}}_{\mathbf{2}}\text{\textsf{Mod}/}%
R\rightarrow \text{\textsf{X}}_{\mathbf{2}}\text{\textsf{Mod}/}S\text{ \ and
}\phi _{\ast }:\text{\textsf{X}}_{\mathbf{2}}\text{\textsf{Mod}/}%
S\rightarrow \text{\textsf{X}}_{\mathbf{2}}\text{\textsf{Mod}/}R
\end{equation*}%
with respect to an algebra morphism $\phi :S\rightarrow R$.  Since
we have pullback object of \textsf{X$_{2} $Mod} along any arrow of
\textsf{k-Alg}, we get the fact
that \textsf{X$_{2} $Mod} is a fibred category and also cofibred
which is the dual of the fibred.

We end with an application which leads to link free $2$-crossed modules with induced $2$-crossed modules.

\subsubsection*{Conventions}

Throughout this paper $k$ will be a fixed commutative ring and $R$
will be a $k$-algebra with identity. All algebras will be commutative and actions will be left
and the right actions in some references will be rewritten by using
left actions.

\section{\protect\bigskip Two-Crossed Modules of Algebras}

Crossed modules of groups were initially defined by Whitehead \cite%
{[w1],[w2]} as models for (homotopy) $2$-types. Conduch\'{e}, \cite{cond},
in $1984$ described the notion of $2$-crossed module as a model for $3$%
-types. Both crossed modules and $2$-crossed modules have been adapted for
use in the context of commutative algebras in \cite{GV,[p2]}.

A crossed module is an algebra morphism $\partial :C\rightarrow R$
with an action of $R$ on $C$ satisfying $\partial \left( r\cdot
c\right) =r\partial \left( c\right) $ and $\partial \left( c\right)
\cdot c^{\prime }=cc^{\prime }$ for all $c,c^{\prime }\in C,r\in R.$
When the first equation is satisfied, $\partial $ is called
pre-crossed module.

If $(C,R,\partial )$ and $(C^{\prime },R^{\prime },\partial ^{\prime })$ are
crossed modules, a morphism,
\[
\left( \theta ,\varphi \right) :(C,R,\partial )\rightarrow (C^{\prime
},R^{\prime },\partial ^{\prime }),
\]%
of crossed modules consists of $k$-algebra homomorphisms $\theta
:C\rightarrow C^{\prime }$ and
\newline
$\varphi :R\rightarrow R^{\prime }$ such that
\[
(i)\ \partial ^{\prime }\theta =\varphi \partial \qquad (ii)\ \theta
(r\cdot c)=\varphi (r)\cdot\theta (c)
\]%
for all $r\in R,c\in C.$ We thus get the category \textsf{XMod }of
crossed modules.

Examples of crossed modules are:

(i) Any ideal, $I$, in $R$ gives an inclusion map $I \rightarrow R,$
which is a crossed module then we will say $\left( I,R,i\right) $ is \ an
ideal pair. In this case, of course, $R$ acts on $I$ by multiplication and
the inclusion homomorphism $i$ makes $\left( I,R,i\right) $ into a crossed
module, an \textquotedblleft inclusion crossed module\textquotedblright .
Conversely, given any crossed module, $\partial : C \rightarrow R$ one easily sees that the image $\partial (C)$ of $C$ is an ideal of $R.$

(ii) Any $R$-module $M$ can be considered as an $R$-algebra with zero
multiplication and hence the zero morphism $0:M\rightarrow R$ sending
everything in $M$ to the zero element of $R$ is a crossed module. Again
conversely, any $(C,R,\partial )$ crossed module, \text{Ker}$\partial $ is an ideal
in $C$ and inherits a natural $R$-module structure from $R$-action on $C.$
Moreover, $\partial (C)$ acts trivially on \text{Ker}$\partial ,$ hence
\text{Ker}$\partial $ has a natural $R/\partial (C)$-module structure.

(iii) Let M$(R)$ be multiplication algebra defined by Mac Lane
\cite{maclane} (see also \cite{lue}) as the set of all multipliers
$\delta:R \rightarrow R$ such that for all $r,r' \in R$, $\delta
(rr')=r\delta \left( r'\right)$ where $R$ is a commutative
$k$-algebra and $Ann\left( R\right) =0$ or $R^{2}=R$. Then $\mu:R
\rightarrow M$(R)$ $ is a crossed module given by $\mu(r)=\delta_r$
with $\delta_r(r')=rr'$ for all $r,r' \in R$. (See \cite{ae} for
details).

(iv) Any epimorphism of algebras $C\rightarrow R$ with the kernel in the
annihilator of $C$ is a crossed module, with $r\in R$ acting on $c\in C$ \
by \ $r\cdot c=\bar{c}c$, where $\bar{c}$ is any element in the pre-image of $%
r$.

Grandje\'{a}n and Vale \cite{GV} have given a definition of
$2$-crossed modules of algebras. The following is an equivalent
formulation of that concept.

A $2$\textit{-crossed module} of $k$-algebras consists of a complex of $P$%
-algebras $L\overset{\partial _{2}}{\rightarrow }M\overset{\partial _{1}}{%
\rightarrow }P$ together with an action of $P$ on all three algebras and a $%
P $-linear mapping%
\begin{equation*}
\left\{ -,-\right\} :M\times M\rightarrow L
\end{equation*}%
which is often called the Peiffer lifting such that the action of $P$ on
itself is by multiplication, $\partial _{2}$ and $\partial _{1}$ are $P$%
-equivariant.

$%
\begin{array}{cl}
\mathbf{PL1}: & \partial _{2}\left\{ m_{0},m_{1}\right\}
=m_{0}m_{1}-\partial _{1}\left( m_{1}\right) \cdot m_{0} \\
\mathbf{PL2}: & \left\{ \partial _{2}\left( l_{0}\right) ,\partial
_{2}\left( l_{1}\right) \right\} =l_{0}l_{1} \\
\mathbf{PL3}: & \left\{ m_{0},m_{1}m_{2}\right\} =\left\{
m_{0}m_{1},m_{2}\right\} +\partial _{1}\left( m_{2}\right) \cdot \left\{
m_{0},m_{1}\right\} \\
\mathbf{PL4}: & \left\{ m,\partial _{2}\left( l\right) \right\} +\left\{
\partial _{2}\left( l\right) ,m\right\} =\partial _{1}\left( m\right) \cdot l
\\
\mathbf{PL5}: & \left\{ m_{0},m_{1}\right\} \cdot p=\left\{ m_{0}\cdot
p,m_{1}\right\} =\left\{ m_{0},m_{1}\cdot p\right\}%
\end{array}%
$

\noindent for all $m,m_{0},m_{1},m_{2}\in M,l,l_{0},l_{1}\in L$ and $p\in P.$

Note that we have not specified that $M$ acts on $L$. We could have done
that as follows: if $m\in M$ and $l\in L$, define
\begin{equation*}
m\cdot l=\left\{ m,\partial _{2}\left( l\right) \right\} .
\end{equation*}%
From this equation $\left( L,M,\partial _{2}\right) $ becomes a crossed
module. We can split $\mathbf{PL4}$ into two pieces:

$\mathbf{PL4:}$%
\begin{equation*}
\begin{array}{crl}
(a) & \left\{ m,\partial _{2}\left( l\right) \right\} & =m\cdot l \\
& \left\{ \partial _{2}\left( l\right) ,m\right\} & =m\cdot l-\partial
_{1}\left( m\right) \cdot l.%
\end{array}%
\end{equation*}

We denote such a $2$-crossed module of algebras by $\left\{ L,M,P,\partial
_{2},\partial _{1}\right\} .$

A morphism of $2$-crossed modules is given by the following diagram

\begin{equation*}
\xymatrix { L \ar[r]^{\partial_{2}} \ar[d]_{f_2} & M
\ar[r]^{\partial_{1}} \ar[d]_{f_1} & P \ar[d]_{f_0} & \\ \ L^\prime
\ar[r]_{\partial^{\prime}_{2}} & M^\prime
\ar[r]_{\partial^{\prime}_{1}} & P^\prime & }
\end{equation*}%
where $f_{0}\partial _{1}=\partial _{1}^{\prime }f_{1},f_{1}\partial
_{2}=\partial _{2}^{\prime }f_{2}$%
\begin{equation*}
\begin{array}{ccc}
f_{1}\left( p\cdot m\right) =\text{ }f_{0}\left( p\right) \cdot \text{ }%
f_{1}\left( m\right) & , & f_{2}\left( p\cdot l\right) =\text{ }f_{0}\left(
p\right) \cdot f_{2}\left( l\right)%
\end{array}%
\end{equation*}%
for all $m\in M,l\in L,p\in P$ and
\begin{equation*}
\left\{ -,-\right\} \left( f_{1}\times f_{1}\right) =f_{2}\left\{ -,-\right\}
\end{equation*}%

We thus get the category of $2$-crossed modules denoting it by \textsf{X}$%
_{2}$\textsf{Mod} and when the morphism $f_{0}$ above is the
identity we will get \textsf{X}$_{2}$\textsf{Mod}$/P$ the category
of $2$-crossed modules over fixed algebra $P.$

Some remarks on Peiffer lifting of $2$-crossed modules are: Suppose we have
a $2$-crossed module%
\begin{equation*}
L\overset{\partial _{2}}{\rightarrow }M\overset{\partial _{1}}{\rightarrow }%
P
\end{equation*}%
with trivial Peiffer lifting. Then

(i) There is an action of $P$ on $L$ and$\ M$ and the $\partial $s are $P$%
-equivariant. (This gives nothing new in our special case.)

(ii) $%
\begin{array}{c}
\left\{ -,-\right\}%
\end{array}%
$ is a lifting of the Peiffer commutator so if $%
\begin{array}{c}
\left\{ m,m^{\prime }\right\} =0,%
\end{array}%
$ the Peiffer identity holds for $\left( M,P,\partial _{1}\right) ,$ i.e.
that is a crossed module.

(iii) if $l,l^{\prime }\in L,$ then $%
\begin{array}{c}
0=\left\{ \partial _{2}l,\partial _{2}l^{\prime }\right\} =ll^{\prime }%
\end{array}%
$

\noindent and,

(iv) as $%
\begin{array}{c}
\left\{ -,-\right\}%
\end{array}%
$ is trivial $%
\begin{array}{c}
\partial _{1}\left( m\right) \cdot l=0%
\end{array}%
$ so $\partial M$ has trivial action on $L.$

\noindent Axioms PL$3$ and PL$5$ vanish.

The above remarks are known for $2$-crossed modules of groups. These are
handled in recent book of Porter in \cite{PORTER}.

\subsection{\textit{Functorial Relations with Some Other Categories}}

\noindent $1.$\ Let $M\overset{\partial }{\rightarrow }P$ be a
pre-crossed module$.$ The Peiffer ideal $\langle M,M\rangle $ is
generated by the
Peiffer commutators%
\begin{equation*}
\langle m,m^{\prime }\rangle =\partial m\cdot m^{\prime }-mm^{\prime
}
\end{equation*}

\noindent for all $m,m^{\prime }\in M.$ The pre-crossed modules in
which all Peiffer commutators are trivial are precisely the crossed
modules. Thus the category of crossed modules is the full
subcategory of the category of
pre-crossed modules whose objects are crossed modules. So we can define the following skeleton functor%
\begin{equation*}
Sk :\text{\textsf{PXMod}}\rightarrow \text{\textsf{X}}_{2}\text{\textsf{Mod}%
}
\end{equation*}%
by $Sk(M,P,\partial _{1})=\left\{ \langle M,M\rangle ,M,P,\partial
_{2},\partial _{1}\right\} $ as a $2$-crossed module with the
Peiffer lifting $\left\{ m,m^{\prime }\right\} =\langle m,m^{\prime
}\rangle $. This
functor has a right adjoint truncation functor:%
\begin{equation*}
Tr :\text{\textsf{X}}_{2}\text{\textsf{Mod}}\rightarrow \text{\textsf{PXMod}%
}
\end{equation*}%
given by $Tr\left\{ L,M,P,\partial _{2},\partial _{1}\right\}
=(M,P,\partial _{1}).$

\noindent $2.$\ Any crossed module gives a $2$-crossed module. If
$\left( M,P,\partial \right) $\textbf{\ }is a crossed module, the
resulting sequence
\begin{equation*}
L\rightarrow M\rightarrow P
\end{equation*}%
is a $2$-crossed module by taking $L=0$. Thus we have
\begin{equation*}
\alpha :\text{\textsf{XMod}}\rightarrow \text{\textsf{X}}_{2}\text{\textsf{%
Mod}}
\end{equation*}%
defined by $\alpha (M,P,\partial )=\left\{ 0,M,P,0,\partial \right\}
$ that is the adjoint of the following functor:
\[
\beta :\mathsf{X}_{2}\mathsf{Mod}\rightarrow \mathsf{XMod}
\]%
given by $\beta \left\{ L,M,P,\partial _{2},\partial _{1}\right\} =\left( M/%
\func{Im}\partial _{2},P,\partial _{1}\right) $ where
$\func{Im}\partial _{2} $ is an ideal of $M.$ $3.$\ The functor
$\delta :\mathsf{XMod}\rightarrow $\textsf{k-Alg} which is
given by $\delta \left( C,R,\partial \right) =R$ has a right adjoint
$\gamma:$ \textsf{k-Alg}$\rightarrow \mathsf{XMod}$, $\gamma\left( A\right)
=\left( A,A,id_{A}\right) .$

\begin{proposition}\label{g1}
Given \textsf{X}$_{2}\mathsf{Mod}\overset{\beta }{\underset{\alpha }{%
\rightleftarrows }}\mathsf{XMod}\overset{\delta }{\underset{\gamma}{%
\rightleftarrows }}$\textsf{k-Alg} adjoint functors as defined
above. Then $\left( \delta \circ \beta ,\alpha \circ \gamma\right) $
is a pair of adjoint functors.
\end{proposition}

\begin{proof}
For $\left\{ L,M,P,\partial _{2},\partial _{1}\right\} \in \ $\textsf{X}$_{2}%
\mathsf{Mod}$ and $R\in \ $\textsf{k-Alg}$,$ we have functorial
isomorphisms:
\[
\begin{array}{rl}
\text{\textsf{k-Alg}}\left( \delta \beta \left( \left\{ L,M,P,\partial
_{2},\partial _{1}\right\} \right) ,R\right)  & \simeq \text{\textsf{XMod}}%
\left( \beta \left\{ L,M,P,\partial _{2},\partial _{1}\right\}
,\gamma\left(
R\right) \right)  \\
& \simeq \textsf{X}_{2}\textsf{Mod}\left( \left\{ L,M,P,\partial
_{2},\partial _{1}\right\} ,\alpha \gamma\left( R\right) \right)
\end{array}%
\]
\end{proof}

Now we will give the construction of pullback and induced 2-crossed
modules. Similar constructions have appeared in several studies on
crossed module of groups, algebras and $2$-crossed modules of groups, e.g.
\cite{aao,[bh],[bhs],[bw1],[bw2]}.
\section{The Pullback Two-Crossed Modules}

The construction of \textquotedblleft change of \
base\textquotedblright \ is well-known in a module theory. The
higher dimension of this had been considered by Porter \cite{[p1]}
and Shammu \cite{[s]}. There are called (co)-induced crossed
modules. The first author and G\"{u}rmen were also deeply analysed
that in \cite{ag}. In this section the functor that is going to be
the right adjoint of the induced $2$-crossed module, the
\textquotedblleft pullback \textquotedblright\ will be defined. This
is an important construction which, given a morphism of algebras
$\phi :S\rightarrow R$, enables us to change of base of $2$-crossed
modules.

\begin{definition}
Given a crossed module $\partial :C\rightarrow R$ and a morphism of $k$%
-algebras $\phi :S\rightarrow R,$ the {\it pullback crossed
module} can be given by

(i) a crossed module $\phi ^{\ast }\left( C,R,\partial \right)
=(\partial ^{\ast }:\phi ^{\ast }(C)\rightarrow S)$

(ii) given
\[
\left( f,\phi \right) :\left( B,S,\mu \right) \longrightarrow
\left( C,R,\partial \right)
\]
crossed module morphism, then there is a unique $\left( f^{\ast
},id_{S}\right) $ crossed module morphism that commutes the
following diagram:

$$\xymatrix@R=40pt@C=40pt{
                &   (B,S,\mu)\ar@{.>}[dl]_{(f^{\ast},id_{S})} \ar[d]^{(f,\phi)}     \\
  (\phi^{\ast}(C),S,\partial^{\ast})  \ar[r]_-{(\phi^{\prime},\phi)} & (C,R,\partial)             }
$$
or more simply as
$$
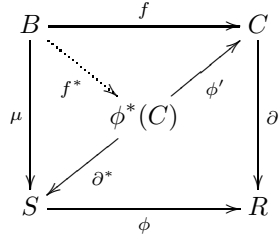
\xymatrix@R=20pt@C=20pt{
  B \ar[dd]_{\mu} \ar[rr]^{f}\ar@{.>}[dr]_{f^{\ast}}
              &  & C \ar[dd]^{\partial}  \\
&\phi^{\ast}(C)\ar[ur]_{\phi^{\prime}}\ar[dl]^{\partial^{\ast}}&
 \\ S  \ar[rr]_{\phi}
               & & R            }$$
where $\phi ^{\ast }(C)=C\times _{R}S=\{(c,s)\mid \partial (c)=\phi
(s)\},\ \partial^{\ast }(c,s)=s,\ $ $\phi ^{\prime }(c,s)=c$ for all
$(c,s)\in \phi ^{\ast }(C)$, and $S$ acts on $\phi ^{\ast }(C)$ via
$\phi $ and the diagonal.
\end{definition}
\begin{definition}
Given a $2$-crossed module $\left\{ C_{2},C_{1},R,\partial _{2},\partial
_{1}\right\} \ $and a morphism of $k$-algebras $\phi :S\rightarrow R,$\ the
pullback $2$-crossed module can be given by
\end{definition}

(i) a $2$-crossed module $%
\begin{array}{c}
\phi ^{\ast }\{C_{2},C_{1},R,\partial _{2},\partial _{1}\}=\{C_{2},\phi
^{\ast }(C_{1}),S,\partial _{2}^{\ast },\partial _{1}^{\ast }\}%
\end{array}%
$

(ii) given any morphism of $2$-crossed modules

\begin{equation*}
(f_{2},f_{1},\phi ):\{B_{2},B_{1},S,\partial _{2}^{\prime },\partial
_{1}^{\prime }\}\rightarrow \left\{ C_{2},C_{1},R,\partial _{2},\partial
_{1}\right\}
\end{equation*}%
there is a unique $(f_{2}^{\ast },f_{1}^{\ast },id_{S})$ $2$-crossed module
morphism that commutes the following diagram:

\begin{equation*}
\xymatrix { & & & (B_2, B_1,S,\partial_{2}^{\prime
},\partial_{1}^{\prime }) \ar@{-->}[ddlll]_{(f_2^*,f_1^*,id_S)}
\ar[dd]^{(f_2,f_1,\phi)} & \\ \ & & & & \\ \ (C_2,\phi^*(C_1), S,
\partial_{2}^{*},\partial_{1}^{*})\ar[rrr]_-{(id_{C_2},\phi^{\prime},\phi)} & & &
(C_2, C_1, R, \partial_{2},\partial_{1})& }
\end{equation*}%
or more simply as
\[
\xymatrix @R=20pt@C=25pt {  B_2  \ar@{-->}[dr]_{f_2^\ast} \ar@/^/[drrr]^{f_2} \ar[dd]_{\partial_2^{\prime }} & & & & \\
\ & C_2 \ar[dd]_{\partial_2^\ast} \ar@{=}[rr]_{id_{C_2}}& & C_2 \ar[dd]^{\partial_2}& \\
\ B_1  \ar@{-->}[dr]_{f_1^\ast} \ar@/^/[drrr]^{f_1} \ar[dd]_{\partial_1^{\prime }} & & & & \\
\ & {\phi^\ast}(C_1) \ar[dd]_{\partial_1^\ast} \ar[rr]_{\phi^{\prime }}& & C_1 \ar[dd]^{\partial_1}& \\
\ S \ar@{=}[dr]_{id_S} \ar@/^/[drrr]^\phi  & & & & \\
\ & S \ar[rr]_\phi & & R.  & }
\]

Let $\phi :S\rightarrow R$ be a morphism of $k$-algebras and let $%
(C_{1},R,\partial _{1})$ be a pre-crossed module. We define
\[
\phi ^{\ast }(C_{1})=C_{1}\times _{R}S=\left\{ \left( c_{1},s\right)
\mid
\partial _{1}\left( c_{1}\right) =\phi \left( s\right) \right\}
\]%
which is usually pullback in the category of algebras. There is a commutative diagram%
\[
\xymatrix{ \phi ^{\ast }(C_{1}) \ar[r]^-{\phi^\prime }
\ar[d]_{\partial_1^\ast} & C_{1} \ar[d]^{\partial_1} &\\
\ S \ar[r]_\phi & R &}
\]
where $\partial _{1}^{\ast }(c_{1},s)=s,$ $\phi ^{\prime
}(c_{1},s)=c_{1}$ and $\partial _{1}^{\ast }$ is $S$-equivariant
with the action $s^{\prime }\cdot (c_{1},s)=({\phi (s^{\prime })}
\cdot c_{1},s^{\prime }s)$ for all $(c_{1},s)\in \phi^{\ast }(C_{1}) $ and $ s\in S$.

So we get a pre-crossed module $\left( \phi ^{\ast }(C_{1}),S,\partial
_{1}^{\ast }\right) $ which is called the pullback pre-crossed module of $%
\left( C_{1},R,\partial _{1}\right) $ along $\phi .$ Then we can define a
pullback of $\partial _{2}:C_{2}\rightarrow C_{1}$ along $\phi
^{\prime }$ as given in the following diagram
\[
\xymatrix{ \phi ^{\ast }(C_{2}) \ar[r]^-{\phi''}
\ar[d]_{\partial_2^\ast} & C_{2} \ar[d]^{\partial_2} &\\
\ \phi ^{\ast }(C_{1}) \ar[r]_-{\phi'} & C_{1} &}
\]
in which
\begin{equation*}
\begin{array}{rl}
\phi ^{\ast }(C_{2}) & =\{(c_{2},(c_{1},s))\mid \partial
_{2}(c_{2})={\phi'}
(c_{1},s)=c_{1},\phi (s)={\partial _{1}}(c_{1})\} \\
& =\{(c_{2},(\partial _{2}(c_{2}),s))\mid \phi (s)=\partial
_{1}\left(
\partial _{2}(c_{2})\right) =0\}\cong C_{2}\times _{C_{1}}\left( Ker{\partial_1} \times Ker\phi \right)%
\end{array}%
\end{equation*}%
for all $c_{2}\in C_{2}$ and $\left( c_{1},s\right) \in \phi ^{\ast }(C_{1}).
$

Since pullback of a pullback is a pullback, we have already constructed the pullback
composition
\begin{equation*}
\phi ^{\ast }(C_{2})\overset{\partial _{2}^{\ast }}{\rightarrow
}\phi ^{\ast }(C_{1})\overset{\partial _{1}^{\ast }}{\rightarrow }S
\end{equation*}%
which is the pullback of $\partial _{1}\partial _{2}=0$ by $\phi .$

On the other hand, we can construct directly the pullback of $\ \partial _{1}\partial
_{2}=0$ by $\phi $ as $\partial :B\rightarrow S$ where $B=\left\{
\left( c_{2},s\right) \mid \phi \left( s\right) =0\right\} \cong
C_{2}\times Ker\phi $ and $\partial (c_{2},s)=s.$ We can define the isomorphism $\Psi :\phi
^{\ast }(C_{2})\rightarrow B$, $\Psi \left( x\right) =\left(
c_{2},s\right) $ where $x=(c_{2},(\partial _{2}(c_{2}),s))\in \phi
^{\ast }(C_{2}).$ So $\phi ^{\ast }(C_{2})\cong B.$

But, we find that the pullback $\phi ^{\ast
}(C_{2})\overset{\partial
_{2}^{\ast }}{\rightarrow }\phi ^{\ast }(C_{1})\overset{\partial _{1}^{\ast }%
}{\rightarrow }S$ is not a complex of $S$-algebras unless $\phi $ is
a monomorphism. To see this, note that for $(c_{2},s)\in C_{2}\times
\ \text{Ker}\phi ,$
\begin{equation*}
\partial _{1}^{\ast }\partial _{2}^{\ast }(c_{2},s)=\partial _{1}^{\ast
}(\partial _{2}c_{2},s)=s.
\end{equation*}%
This last expression is equal to $0$ if $\phi $ is a monomorphism.
So $\phi ^{\ast }(C_{2})\cong C_{2}.$ Thus, we can give the pullback $2$-crossed module of $C_{2}\overset{\partial _{2}}{%
\rightarrow }C_{1}\overset{\partial _{1}}{\rightarrow }R$ along
$\phi $ as follows.

\begin{proposition}
\label{ug2}If $C_{2}\overset{\partial _{2}}{\rightarrow }C_{1}\overset{\partial _{1}}{%
\rightarrow }R$ is a $2$-crossed module and if $\phi :S\rightarrow
R$ is a monomorphism of k-algebras then
\begin{equation*}
C_{2}\overset{\partial _{2}^{\ast }}{\rightarrow }\phi ^{\ast }(C_{1})%
\overset{\partial _{1}^{\ast }}{\rightarrow }S
\end{equation*}%
is a pullback $\ 2$-crossed module where $\partial _{2}^{\ast
}(c_{2})=(\partial _{2}(c_{2}),0)$ and $\partial _{1}^{\ast
}(c_{1},s)=s$
and the action of $S$ on $\phi ^{\ast }\left( C_{1}\right) $ and $C_{2}$ by $%
{s} \cdot (c_{1},s^{\prime })=\left({\phi \left( s \right) } \cdot
c_{1},ss^{\prime }\right) $ and ${s} \cdot c_{2}={\phi (s)} \cdot c_{2}$
respectively.
\end{proposition}

\begin{proof}
Since%
\begin{equation*}
\begin{array}{rcl}
\partial _{2}^{\ast }(s\cdot c_{2}) & = & \left( \partial _{2}\left( s\cdot
c_{2}\right) ,0\right)  \\
& = & \left( \partial _{2}(\phi \left( s\right) \cdot c_{2}),0\right)  \\
& = & \left( \phi (s)\cdot \partial _{2}\left( c_{2}\right) ,0\right)  \\
& = & s\cdot \left( \partial _{2}\left( c_{2}\right) ,0\right)  \\
& = & s\cdot \partial _{2}^{\ast }(c_{2}),%
\end{array}
\end{equation*}%
$\partial _{2}^{\ast }$ is $S$-equivariant. Also, we have seen above that $\partial _{1}^{\ast }$ is $S$-equivariant and
$C_{2}\rightarrow \phi ^{\ast }(C_{1})\rightarrow S$ is a complex of $S$-algebras.

The Peiffer lifting%
\begin{equation*}
\left\{ -,-\right\} :\phi ^{\ast }(C_{1})\times \phi ^{\ast
}(C_{1})\rightarrow C_{2}
\end{equation*}%
is given by $\left\{ \left( c_{1},s_{1}\right) \left( c_{1}^{\prime
},s_{1}^{\prime }\right) \right\} =\left\{ c_{1},c_{1}^{\prime
}\right\} .$

\textbf{PL1:}%
\begin{equation*}
\begin{array}{rcl}
\left( c_{1},s_{1}\right) \left( c_{1}^{\prime },s_{1}^{\prime }\right)
-\left( c_{1},s_{1}\right) \cdot \partial _{1}^{\ast }\left( c_{1}^{\prime
},s_{1}^{\prime }\right) & = & (c_{1}c_{1}^{\prime },s_{1}s_{1}^{\prime
})-\left( c_{1},s_{1}\right) \cdot s_{1}^{\prime } \\
& = & (c_{1}c_{1}^{\prime },s_{1}s_{1}^{\prime })-(c_{1}\cdot \phi \left(
s_{1}^{\prime }\right) ,s_{1}s_{1}^{\prime }) \\
& = & (c_{1}c_{1}^{\prime }-c_{1}\cdot \phi \left( s_{1}^{\prime }\right) ,0)
\\
& = & (c_{1}c_{1}^{\prime }-c_{1}\cdot \partial _{1}\left( c_{1}^{\prime
}\right) ,0) \\
& = & \left( \partial _{2}\left\{ c_{1},c_{1}^{\prime }\right\} ,0\right) \\
& = & \partial _{2}^{\ast }\left( \left\{ c_{1},c_{1}^{\prime }\right\}
\right) \\
& = & \partial _{2}^{\ast }\left\{ \left( c_{1},s_{1}\right) ,\left(
c_{1}^{\prime },s_{1}^{\prime }\right) \right\} .%
\end{array}%
\end{equation*}

\textbf{PL2:}%
\begin{equation*}
\begin{array}{rcl}
\left\{ \partial _{2}^{\ast }\left( c_{2}\right) ,\partial _{2}^{\ast
}\left( c_{2}^{\prime }\right) \right\} & = & \left\{ \left( \partial
_{2}\left( c_{2}\right) ,0\right) ,\left( \partial _{2}\left( c_{2}^{\prime
}\right) ,0\right) \right\} \\
& = & \left\{ \partial _{2}\left( c_{2}\right) ,\partial _{2}\left(
c_{2}^{\prime }\right) \right\} \\
& = & c_{2}c_{2}^{\prime }.%
\end{array}%
\end{equation*}

The rest of axioms of $2$-crossed module is given in appendix.%
\begin{equation*}
(id_{C_{2}},\phi ^{\prime },\phi ):\{C_{2},\phi ^{\ast }(C_{1}),S,\partial
_{2}^{\ast },\partial _{1}^{\ast }\}\rightarrow \left\{
C_{2},C_{1},R,\partial _{2},\partial _{1}\right\}
\end{equation*}%
or diagrammatically,%
\begin{equation*}
\xymatrix { C_2 \ar[d]_{\partial_{2}^*} \ar@{=}[r]^{id_{C_2}} & C_2
\ar[d]^{\partial_2} & \\ \ \phi^*(C_1) \ar[d]_{\partial_{1}^*}
\ar[r]^-{\phi^{\prime }} & C_1 \ar[d]^{\partial_1}& \\ \ S
\ar[r]_\phi & R & }
\end{equation*}%
is a morphism of $2$-crossed modules. (See appendix.)

Suppose that%
\begin{equation*}
(f_{2},f_{1},\phi ):\{B_{2},B_{1},S,\partial _{2}^{\prime },\partial
_{1}^{\prime }\}\rightarrow \left\{ C_{2},C_{1},R,\partial _{2},\partial
_{1}\right\}
\end{equation*}%
is any $2$-crossed module morphism%
\begin{equation*}
\xymatrix { B_2 \ar[r]^{\partial^{\prime }_{2}} \ar[d]_{f_2} & B_1
\ar[d]_{f_1} \ar[r]^{\partial^{\prime }_{1}} & S \ar[d]_\phi & \\ \
C_2 \ar[r]_{\partial_{2}} & C_1 \ar[r]_{\partial_{1}} & R. & \\ }
\end{equation*}%
Then we will show that there is a unique $2$-crossed module morphism%
\begin{equation*}
(f_{2}^{\ast },f_{1}^{\ast },id_{S}):\{B_{2},B_{1},S,\partial _{2}^{\prime
},\partial _{1}^{\prime }\}\rightarrow \{C_{2},\phi ^{\ast
}(C_{1}),S,\partial _{2}^{\ast },\partial _{1}^{\ast }\}
\end{equation*}%
\begin{equation*}
\xymatrix { B_2 \ar[r]^{\partial^{\prime }_{2}} \ar[d]_{f_2^*} & B_1
\ar[d]_{f_1^*} \ar[r]^{\partial^{\prime }_{1}} & S \ar@{=}[d]_{id_S} & \\
\  C_2 \ar[r]_{\partial_2} & \phi^*(C_1) \ar[r]_{\partial_1} & S &
\\ }
\end{equation*}%
where $f_{2}^{\ast }(b_{2})=f_{2}(b_{2})$ and $f_{1}^{\ast }(b_{1})=\left(
f_{1}(b_{1}),\partial _{1}^{\prime }(b_{1})\right) $ which is an element in $%
\phi ^{\ast }(C_{1}).$ First, let us check that $(f_{2}^{\ast },f_{1}^{\ast
},id_{S})$ is a morphism of $2$-crossed modules. For $b_{1}\in
B_{1},b_{2}\in B_{2},s\in S$%
\begin{equation*}
\begin{array}{ccl}
f_{2}^{\ast }(s\cdot b_{2}) & = & f_{2}(s\cdot b_{2}) \\
& = & \phi \left( s\right) \cdot f_{2}(b_{2}) \\
& = & \phi \left( s\right) \cdot f_{2}^{\ast }(b_{2}) \\
& = & s\cdot f_{2}^{\ast }(b_{2}) \\
& = & id_{S}\left( s\right) \cdot f_{2}^{\ast }(b_{2}),%
\end{array}%
\end{equation*}%
similarly $f_{1}^{\ast }(s\cdot b_{1})=$ $id_{S}\left( s\right) \cdot
f_{1}^{\ast }(b_{1}),$ also the above diagram is commutative and%
\[
\begin{array}{ccl}
\left\{ -,-\right\} \left( f_{1}^{\ast }\times f_{1}^{\ast }\right)
(b_{1},b_{1}^{\prime }) & = & \left\{ -,-\right\} \left( f_{1}^{\ast }\left(
b_{1}\right) ,f_{1}^{\ast }\left( b_{1}^{\prime }\right) \right)  \\
& = & \left\{ -,-\right\} \left( \left( f_{1}\left( b_{1}\right) ,\partial
_{1}^{\prime }\left( b_{1}\right) \right) ,\left( f_{1}\left( b_{1}^{\prime
}\right) ,\partial _{1}^{\prime }\left( b_{1}^{\prime }\right) \right)
\right)  \\
& = & \left\{ \left( f_{1}\left( b_{1}\right) ,\partial _{1}^{\prime }\left(
b_{1}\right) \right) ,\left( f_{1}\left( b_{1}^{\prime }\right) ,\partial
_{1}^{\prime }\left( b_{1}^{\prime }\right) \right) \right\}  \\
& = & \left\{ f_{1}\left( b_{1}\right) ,f_{1}\left( b_{1}^{\prime }\right)
\right\}  \\
& = & \left\{ -,-\right\} \left( f_{1}\left( b_{1}\right) ,f_{1}\left(
b_{1}^{\prime }\right) \right)  \\
& = & \left\{ -,-\right\} \left( f_{1}\times f_{1}\right)
(b_{1},b_{1}^{\prime }) \\
& = & f_{2}\left\{ -,-\right\} (b_{1},b_{1}^{\prime }) \\
& = & f_{2}\left( \left\{ b_{1},b_{1}^{\prime }\right\} \right)  \\
& = & f_{2}^{\ast }\left( \left\{ b_{1},b_{1}^{\prime }\right\} \right)  \\
& = & f_{2}^{\ast }\left\{ -,-\right\} (b_{1},b_{1}^{\prime })%
\end{array}%
\]
for all $b_{1},b_{1}^{\prime }\in B_{1}.$ So $(f_{2}^{\ast },f_{1}^{\ast
},id_{S})$ is a morphism of 2-crossed modules.

Furthermore; following equations are easily verified:

\begin{equation*}
id_{C_{2}}f_{2}^{\ast }=f_{2\qquad }\text{and}\qquad \phi ^{\prime
}f_{1}^{\ast }=f_{1}.
\end{equation*}
\end{proof}

\noindent Thus we get a functor
\begin{equation*}
\phi ^{\ast }:\text{\textsf{X}}_{2}\text{\textsf{Mod}}/R\rightarrow \text{%
\textsf{X}}_{2}\text{\textsf{Mod}}/S
\end{equation*}%
which gives our pullback $2$-crossed module.
\begin{remark}\label{natural}
These functors have the property that for any monomorphisms $\phi$
and $\phi ^{\prime }$ there are natural
equivalences $\phi ^{\ast }\phi ^{\prime \ast }\simeq \left( \phi ^{\prime
}\phi \right) ^{\ast }.$
\end{remark}
\subsection{The Examples of Pullback Two-Crossed Modules}

Given $2$-crossed module $%
\begin{array}{c}
\left\{ \left\{ 0\right\} ,I,R,0,i\right\}%
\end{array}%
$where $i$ is an inclusion of an ideal. The pullback $2$-crossed module is%
\begin{equation*}
\begin{array}{ccl}
\phi ^{\ast }\left\{ \left\{ 0\right\} ,I,R,0,i\right\} & = & \left\{
\left\{ 0\right\} ,\phi ^{\ast }\left( I\right) ,S,\partial _{2}^{\ast
},\partial _{1}^{\ast }\right\} \\
& = & \left\{ \left\{ 0\right\} ,\phi ^{-1}\left( I\right) ,S,\partial
_{2}^{\ast },\partial _{1}^{\ast }\right\}%
\end{array}%
\end{equation*}%
as,%
\begin{equation*}
\begin{array}{ccl}
\phi ^{\ast }\left( I\right) & = & \left\{ (a,s)\mid \phi \left( s\right)
=i(a)=a,s\in S,a\in I\right\} \\
& \cong & \left\{ s\in S\mid \phi \left( s\right) =a\right\} =\phi
^{-1}\left( I\right) \unlhd S.%
\end{array}%
\end{equation*}

\noindent The pullback diagram is

\begin{equation*}
\xymatrix { \{0\} \ar[d]_{\partial_{2}^{*}=0} \ar@{=}[r] & \{0\} \ar[d]^{0} & \\
\ \phi^{-1}(I) \ar[d]_{\partial_{1}^{*}} \ar[r] & I \ar[d]^i & \\ \ S
\ar[r]_\phi & R. & }
\end{equation*}

Particularly if $I=\left\{ 0\right\} ,$ since $\phi $ is monomorphism we get
\begin{equation*}
\phi ^{\ast }\left( \left\{ 0\right\} \right) \cong \left\{ s\in S\mid \phi
\left( s\right) =0\right\} =\text{ker}\phi \cong \{0\}
\end{equation*}%
and so $\left\{ \left\{ 0\right\} ,\left\{ 0\right\} ,S,0,0 \right\}
\ $is a pullback $2$-crossed module.

Also if $\phi $ is an isomorphism and $I=R,$ then $\phi ^{\ast }\left(
R\right) =R\times S.$

Similarly when we consider examples given in Section $1$, the following
diagrams are pullbacks.%
\begin{equation*}
\xymatrix { \{0\} \ar[d] \ar@{=}[r] & \{0\}
\ar[d] & \\ \ \phi^{*}(R) \ar[d] \ar[r] & R
\ar[d] & \\ M (S) \ar[r]_-{M({\phi})} & M (R) & }\xymatrix { \{0\}
\ar[d] \ar@{=}[r] & \{0\} \ar[d] & \\ \ M \times
Ker\phi \ar[d] \ar[r] & M \ar[d] & \\ \ S
\ar[r]_\phi & R & }
\end{equation*}
\section{The Induced Two-Crossed Modules}\label{induced}

We will define a functor $\phi_{\ast }$ left adjoint to the pullback
$\phi^{\ast }$ given in the previous section. The
\textquotedblleft induced $2$-crossed module\textquotedblright \
functor $\phi_{\ast }$ is defined by the following universal property, developing works in \cite{ae} and \cite{[p1]}.

\begin{definition}
For any crossed module $\mathcal{D}=(D,S,\partial )$ and any homomorphism $%
\phi :S\rightarrow R$ the crossed module $induced$ by $\phi $ from
$\partial $ should be given by:

$(i)$ a crossed module $\phi _{\ast }\left( \mathcal{D}\right)
=(\phi _{\ast }\left( D\right) $,$R,\phi _{\ast }\partial ),$

$(ii)$ a morphism of crossed modules $\left( f,\phi \right) :\mathcal{D}%
\rightarrow \phi _{\ast }\left( \mathcal{D}\right) ,$ satisfying the
dual
universal property that for any morphism of crossed modules%
\begin{equation*}
(h,\phi ):\mathcal{D\rightarrow B}
\end{equation*}%
there is a unique morphism of crossed modules $h^{\prime }:\phi
_{\ast
}(D)\rightarrow B$ such that the diagram%
\begin{equation*}
\xymatrix { & & B \ar@/^/[ddl]^v & \\ D \ar[r]_f \ar@/^/[urr]^h
\ar[d]^\partial & \phi _*(D) \ar[d]_{\phi _*\partial}
\ar@{-->}[ur]^{h^\prime}& & \\ S \ar[r]^\phi & R & & }
\end{equation*}%
commutes.
\end{definition}

\begin{definition}
For any $2$-crossed module $D_{2}\overset{\partial _{2}}{\rightarrow }D_{1}%
\overset{\partial _{1}}{\rightarrow }S$ and a morphism $\phi :S\rightarrow R$
of $k$-algebras, the induced $2$-crossed can be given by
\end{definition}

(i) a $2$-crossed module $\phi _{\ast }\left\{ D_{2},D_{1},S,\partial
_{2},\partial _{1}\right\} =\left\{ \phi _{\ast }(D_{2}),\phi _{\ast
}(D_{1}),R,\partial _{2_{\ast }},\partial _{1_{\ast }}\right\} $

(ii) given any morphism of $2$-crossed modules%
\begin{equation*}
(f_{2},f_{1},\phi ):\left\{ D_{2},D_{1},S,\partial _{2},\partial
_{1}\right\} \rightarrow \{B_{2},B_{1},R,\partial _{2}^{\prime
},\partial _{1}^{\prime }\}
\end{equation*}%
then there is a unique $(f_{2_\ast },f_{1_\ast },id_{R})\ 2$-crossed module
morphism that commutes the following diagram:
\begin{equation*}
\xymatrix { (D_2, D_1, S, \partial_2, \partial_1)
\ar[ddrrr]^{(\phi^{\prime\prime},\phi^{\prime},\phi)}
\ar[dd]_{(f_2,f_1,\phi)} & & & & \\ \ & & & & \\ \ (B_2, B_1, R,
\partial_2^{\prime},
\partial_1^{\prime}) & & & (\phi_\ast(D_2), \phi_\ast(D_1), R,
\partial _{2_{\ast }}, \partial _{1_{\ast }})
\ar@{-->}[lll]^{(f_{2_\ast},f_{1_\ast},id_R)}& }
\end{equation*}%
or more simply as

\[
\xymatrix @R=20pt@C=25pt { & & & B_2  \ar[dd]^{\partial _{2}^{\prime
}} & \\
\ D_2  \ar[dd]_{\partial_2} \ar[rr]_{\phi^{\prime \prime}} \ar@/^/[urrr]^{f_2} &  & \phi_\ast(D_2) \ar@{-->}[ur]_{f_{2_\ast}} \ar[dd]^{\partial_{2_\ast}} &  & \\
\ & & & B_1  \ar[dd]^{\partial _{1}^{\prime
}} & \\
\ D_1 \ar[dd]_{\partial_1} \ar[rr]_{\phi^\prime} \ar@/^/[urrr]^{f_1} &  & \phi_\ast(D_1) \ar@{-->}[ur]_{f_{1_\ast}} \ar[dd]^{\partial_{1_\ast}} & & \\
\  & & & R  & \\
\  S \ar@/^/[urrr]^\phi \ar[rr]_\phi &  & R. \ar@{=}[ur]_{id_R}& & }
\]

 \begin{proposition}
\label{ug3}Let $D_{2}\overset{\partial _{2}}{\rightarrow }D_{1}\overset{%
\partial _{1}}{\rightarrow }S$ be a $2$-crossed module and $\phi
:S\rightarrow R$ be a morphism of $k$-algebras. Then $\phi _{\ast }(D_{2})%
\overset{\partial _{2_{\ast }}}{\rightarrow }\phi _{\ast }(D_{1})\overset{%
\partial _{1_{\ast }}}{\rightarrow }R$ is the induced $2$-crossed module
where $\phi _{\ast }(D_{1})$ is generated as an algebra, by the set $%
D_{1}\times R$ with defining relations%
\begin{equation*}
\begin{array}{l}
(d_{1},r_{1})(d_{1}^{\prime },r_{1}^{\prime })=(d_{1}d_{1}^{\prime
},r_{1}r_{1}^{\prime }), \\
(d_{1},r)-(d_{1}^{\prime },r)=(d_{1}-d_{1}^{\prime },r), \\
\left( s.d_{1},r\right) =(d_{1},\phi \left( s\right) r)%
\end{array}%
\end{equation*}%
and $\phi _{\ast }(D_{2})$ is generated as an algebra, by the set $%
D_{2}\times R$ with defining relations%
\begin{equation*}
\begin{array}{l}
\left\{ \left( d_{1},r_{1}\right) +\left( d_{1}^{\prime },r_{1}^{\prime
}\right) ,(d_{1}^{\prime \prime },r_{1}^{\prime \prime })\right\} =\left\{
\left( d_{1},r_{1}\right) ,(d_{1}^{\prime \prime },r_{1}^{\prime \prime
})\right\} +\left\{ \left( d_{1}^{\prime },r_{1}^{\prime }\right)
,(d_{1}^{\prime \prime },r_{1}^{\prime \prime })\right\} , \\
\left\{ \left( d_{1},r_{1}\right) ,\left( d_{1}^{\prime },r_{1}^{\prime
}\right) +(d_{1}^{\prime \prime },r_{1}^{\prime \prime })\right\} =\left\{
\left( d_{1},r_{1}\right) ,\left( d_{1}^{\prime },r_{1}^{\prime }\right)
\right\} +\left\{ \left( d_{1},r_{1}\right) ,(d_{1}^{\prime \prime
},r_{1}^{\prime \prime })\right\} , \\
(d_{2},r_{2})(d_{2}^{\prime },r_{2}^{\prime })=(d_{2}d_{2}^{\prime
},r_{2}r_{2}^{\prime }), \\
(d_{2},r)+(d_{2}^{\prime },r)=(d_{2}+d_{2}^{\prime },r), \\
(s\cdot d_{2},r)=(d_{2},\phi (s)r)%
\end{array}%
\end{equation*}%
for any $d_{1},d_{1}^{\prime },d_{1}^{\prime \prime }\in D_{1},$ $%
d_{2},d_{2}^{\prime }\in D_{2},s\in S$ and$\ r,r_{1},r_{1}^{\prime
},r_{1}^{\prime \prime },r_{2},r_{2}^{\prime }\in R.$ The morphism $\partial
_{2_{\ast }}:\phi _{\ast }(D_{2})\rightarrow \phi _{\ast }(D_{1})$ is given
by $\partial _{2_{\ast }}(d_{2},r)=(\partial _{2}(d_{2}),r)$ the action of $%
\phi _{\ast }(D_{1})$ on $\phi _{\ast }(D_{2})$ by $(d_{1},r_{1})\cdot
(d_{2},r_{2})=(d_{1}\cdot d_{2},r_{2}),$ and the morphism $\partial
_{1_{\ast }}:\phi _{\ast }(D_{1})\rightarrow R$ is given by $\partial
_{1_{\ast }}(d_{1},r)=\phi \partial _{1}(d_{1})r,$ the action of $R$ on $%
\phi _{\ast }(D_{1})$ and $\phi _{\ast }(D_{2})$ by $r\cdot
(d_{1},r_{1})=(d_{1},rr_{1})$ and $r\cdot (d_{2},r^{\prime
})=(d_{2},rr^{\prime })$ respectively$.$
\end{proposition}

\begin{proof}
(i) As $\partial _{1_{\ast }}\left( \partial _{2_{\ast }}(d_{2},r)\right)
=\partial _{1_{\ast }}(\partial _{2}(d_{2}),r)=\phi \left( \partial
_{1}\left( \partial _{2}(d_{2})\right) \right) r=\phi (0)r=0,$%
\begin{equation*}
\phi _{\ast }(D_{2})\overset{\partial _{2_{\ast }}}{\rightarrow }\phi _{\ast
}(D_{1})\overset{\partial _{1_{\ast }}}{\rightarrow }R
\end{equation*}%
is a complex of $k$-algebras. The Peiffer lifting%
\begin{equation*}
\left\{ -,-\right\} :\phi _{\ast }(D_{1})\times \phi _{\ast
}(D_{1})\rightarrow \phi _{\ast }(D_{2})
\end{equation*}%
is given by $\left\{ \left( d_{1},r_{1}\right) ,\left( d_{1}^{\prime
},r_{1}^{\prime }\right) \right\} =\left( \left\{
d_{1},d_{1}^{\prime }\right\} ,r_{1}r_{1}^{\prime }\right)$ for all
$ \left( d_{1},r_{1}\right),\left( d_{1}^{\prime },r_{1}^{\prime
}\right) \in \phi _{\ast }(D_{1}).$

\textbf{PL1:}%
\begin{equation*}
\begin{array}{rcl}
\partial _{2_{\ast }}\left\{ \left( d_{1},r_{1}\right) ,\left( d_{1}^{\prime
},r_{1}^{\prime }\right) \right\} & = & \partial _{2_{\ast }}\left( \left\{
d_{1},d_{1}^{\prime }\right\} ,r_{1}r_{1}^{\prime }\right) \\
& = & \left( \partial _{2}\left\{ d_{1},d_{1}^{\prime }\right\}
,r_{1}r_{1}^{\prime }\right) \\
& = & (d_{1}d_{1}^{\prime }-d_{1}\cdot \partial _{1}(d_{1}^{\prime
}),r_{1}r_{1}^{\prime }) \\
& = & (d_{1}d_{1}^{\prime },r_{1}r_{1}^{\prime })-(d_{1}\cdot \partial
_{1}(d_{1}^{\prime }),r_{1}r_{1}^{\prime }) \\
& = & (d_{1}d_{1}^{\prime },r_{1}r_{1}^{\prime })-(d_{1},\phi \left(
\partial _{1}(d_{1}^{\prime })\right) r_{1}r_{1}^{\prime }) \\
& = & (d_{1}d_{1}^{\prime },r_{1}r_{1}^{\prime })-(d_{1},r_{1})\cdot \left(
\phi \left( \partial _{1}(d_{1}^{\prime })\right) r_{1}^{\prime }\right) \\
& = & (d_{1}d_{1}^{\prime },r_{1}r_{1}^{\prime })-(d_{1},r_{1})\cdot
\partial _{1_{\ast }}\left( d_{1}^{\prime },r_{1}^{\prime }\right) \\
& = & \left( d_{1},r_{1}\right) \left( d_{1}^{\prime },r_{1}^{\prime
}\right) -\left( d_{1},r_{1}\right) \cdot \partial _{1_{\ast }}\left(
d_{1}^{\prime },r_{1}^{\prime }\right) .%
\end{array}%
\end{equation*}

\textbf{PL2:}%
\begin{equation*}
\begin{array}{rcl}
\left\{ \partial _{2_{\ast }}(d_{2},r_{2}),\partial _{2_{\ast
}}(d_{2}^{\prime },r_{2}^{\prime })\right\} & = & \left\{ (\partial
_{2}\left( d_{2}),r_{2}\right), (\partial _{2}\left( d_{2}^{\prime
}),r_{2}^{\prime }\right) \right\} \\
& = & \left( \left\{ \partial _{2}\left( d_{2}\right) ,\partial
_{2}(d_{2}^{\prime })\right\} ,r_{2}r_{2}^{\prime }\right) \\
& = & \left( d_{2}d_{2}^{\prime },r_{2}r_{2}^{\prime }\right) \\
& = & (d_{2},r_{2})(d_{2}^{\prime },r_{2}^{\prime })%
\end{array}%
\end{equation*}%
\noindent for all $ \left( d_{1},r_{1}\right),\left( d_{1}^{\prime
},r_{1}^{\prime }\right) \in \phi _{\ast }(D_{1}),\ \left(
d_{2},r_{2}\right),\left( d_{2}^{\prime },r_{2}^{\prime }\right) \in
\phi _{\ast }(D_{2}).$

The rest of axioms of $2$-crossed module is given in appendix.

(ii)\ It is clear that
\begin{equation*}
(\phi ^{\prime \prime },\phi ^{\prime },\phi ):\left\{
D_{2},D_{1},S,\partial _{2},\partial _{1}\right\} \rightarrow \left\{ \phi
_{\ast }(D_{2}),\phi _{\ast }(D_{1}),R,\partial _{2_{\ast }},\partial
_{1_{\ast }}\right\}
\end{equation*}%
or diagrammatically,%
\begin{equation*}
\xymatrix { D_2 \ar[r]^-{\phi^{\prime\prime}} \ar[d]_{\partial_{2}}&
\phi_*(D_2) \ar[d]^{\partial _{2_{\ast }}} & \\ \ D_1 \ar[r]^-{\phi^\prime}
\ar[d]_{\partial_{1}}& \phi_*(D_1) \ar[d]^{\partial _{1_{\ast }}} & \\ \ S
\ar[r]^-{\phi} & R & }
\end{equation*}%
is a morphism of $2$-crossed modules.

Suppose that%
\begin{equation*}
(f_{2},f_{1},\phi ):\left\{ D_{2},D_{1},S,\partial _{2},\partial
_{1}\right\} \rightarrow \left\{ B_{2},B_{1},R,\partial _{2}^{\prime
},\partial _{1}^{\prime }\right\}
\end{equation*}%
is any $2$-crossed module morphism. Then we will show that there is a $2$%
-crossed module morphism%
\begin{equation*}
(f_{2_{_\ast }},f_{1_{\ast }},id_{R}):\left\{ \phi _{\ast }(D_{2}),\phi
_{\ast }(D_{1}),R,\partial _{2_{\ast }},\partial _{1_{\ast }}\right\}
\rightarrow \left\{ B_{2},B_{1},R,\partial _{2}^{\prime },\partial
_{1}^{\prime }\right\}
\end{equation*}%
\begin{equation*}
\xymatrix { \phi_*(D_2) \ar[r]^{\partial _{2_{\ast }}}
\ar[d]_{f_2{_\ast}} & \phi_*(D_1) \ar[r]^{\partial _{1_{\ast }}}
\ar[d]_{f_1{_\ast}} & R \ar@{=}[d]_{id_R} & \\ \ B_2
\ar[r]_{\partial_{2}^{\prime }} & B_1 \ar[r]_{\partial_{1}^{\prime
}} & R & }
\end{equation*}%
where $f_{2_{\ast }}(d_{2},r_{2})=r_{2}\cdot f_{2}(d_{2})$\ and $f_{1_{\ast
}}(d_{1},r_{1})=r_{1}\cdot f_{1}(d_{1}).$ First we will check that $%
(f_{2_{\ast }},f_{1_{\ast }},id_{R})\ $is a $2$-crossed module morphism. We
can see this easily as follows:%
\begin{equation*}
\begin{array}{ccl}
f_{2_{\ast }}\left( r_{2}^{\prime }\cdot (d_{2},r_{2})\right) & = &
f_{2_{\ast }}(d_{2},r_{2}^{\prime }r_{2}) \\
& = & r_{2}^{\prime }\cdot \left( r_{2}\cdot f_{2}(d_{2})\right) \\
& = & r_{2}^{\prime }\cdot f_{2_{\ast }}(d_{2},r_{2}).%
\end{array}%
\end{equation*}%
Similarly $f_{1_{\ast }}\left( r_{1}^{\prime }\cdot (d_{1},r_{1})\right)
=r_{1}^{\prime }\cdot f_{1_{\ast }}(d_{1},r_{1}),$
\begin{equation*}
\begin{array}{ccl}
\left( f_{1_\ast }\partial _{2_\ast }\right) \left(
d_{2},r_{2}\right) & = &
f_{1_\ast }\left( \partial _{2_\ast }\left( d_{2},r_{2}\right) \right) \\
& = & f_{1_\ast }\left( \partial _{2}\left( d_{2}\right) ,r_{2}\right) \\
& = & r_{2}\cdot \left( f_{1}(\partial _{2}(d_{2}))\right) \\
& = & r_{2}\cdot \left( \partial _{2}^{\prime }\left( f_{2}\left(
d_{2}\right) \right) \right) \\
& = & \partial _{2}^{\prime }\left( r_{2}\cdot f_{2}\left( d_{2}\right)
\right) \\
& = & \left( \partial _{2}^{\prime }f_{2_{\ast }}\right) \left(
d_{2},r_{2}\right)%
\end{array}%
\end{equation*}%
and $\partial _{1}^{\prime }f_{1\ast }=id_{R}\partial _{1_{\ast }}$ for all $%
\left( d_{1},r_{1}\right) \in \phi _{\ast }(D_{1}),\left( d_{2},r_{2}\right)
\in \phi _{\ast }(D_{2}),r_{1}^{\prime },r_{2}^{\prime }\in R$ \ and%
\begin{equation*}
\left\{ -,-\right\} \left( f_{1\ast }\times f_{1\ast }\right) =f_{2}^{\ast
}\left\{ -,-\right\} .
\end{equation*}
\end{proof}

So we get the induced $2$-crossed module functor
\[
\phi _{\ast }:\mathsf{X}_{2}\mathsf{Mod}/S\rightarrow \mathsf{X}_{2}\mathsf{%
Mod}/R.
\]

We can give the following naturality condition for $\phi _{\ast }$ similar to remark \ref{natural}.
\begin{remark}
They satisfy the \textquotedblleft naturality condition\textquotedblright\ that
there is a natural equivalence of functors%
\[
\phi _{\ast }^{\prime }\phi _{\ast }\simeq \left( \phi ^{\prime }\phi
\right) _{\ast }.
\]

\end{remark}

\begin{theorem}\label{ug6}
For any morphism of $k$-algebras $\phi :S\rightarrow R,$ $\phi _{\ast }$ is
the left adjoint of $\phi ^{\ast }.$
\end{theorem}
\begin{proof}
From the naturality conditions given earlier, it is immediate that for
any $2$-crossed modules $\mathcal{D}=\left\{ D_{2},D_{1},S,\partial
_{2},\partial _{1}\right\} $ and $\mathcal{B}=\left\{ B_{2},B_{1},R,\partial
_{2},\partial _{1}\right\} $ there are bijections
\[
\left( \mathsf{X}_{2}\mathsf{Mod/}S\right) \left( \mathcal{D},\phi ^{\ast }\left(
\mathcal{B}\right) \right) \cong \{(f_{2},f_{1})\mid (f_{2},f_{1},\phi ):\mathcal{D}%
\rightarrow \mathcal{B}\text{ is a morphism in }\mathsf{X}_{2}\mathsf{Mod\}}
\]%
as proved in proposition \ref{ug2}, and%
\[
\left( \mathsf{X}_{2}\mathsf{Mod/}R\right) \left( \phi _{\ast }\left(
\mathcal{D}\right) ,\mathcal{B}\right) \cong \{(f_{2},f_{1})\mid
(f_{2},f_{1},\phi ):\mathcal{D}\rightarrow \mathcal{B}\text{ is a morphism
in }\mathsf{X}_{2}\mathsf{Mod\}}
\]%
as given in the proposition  \ref{ug3}.

Their composition gives the bijection needed for adjointness.
\end{proof}

Next if $\phi :S\rightarrow R,$ is an epimorphism the induced $2$-crossed
module has a simpler description.

\begin{proposition}
\label{ug4}Let $D_{2}\overset{\partial _{2}}{\rightarrow }D_{1}\overset{%
\partial _{1}}{\rightarrow }S$ is a $2$-crossed module, $\phi :S\rightarrow
R $ is an epimorphism with Ker$\phi =K.$ Then
\begin{equation*}
\phi _{\ast }(D_{2})\cong D_{2}/KD_{2}\qquad \text{and}\qquad \phi _{\ast
}(D_{1})\cong D_{1}/KD_{1},
\end{equation*}%
where $KD_{2}$ denotes the ideal of $D_{2}$ generated by $\left\{ k\cdot
d_{2}\mid k\in K,d_{2}\in D_{2}\right\} $ and $KD_{1}$ denotes the ideal of $%
D_{1}$ generated by $\left\{ k\cdot d_{1}\mid k\in K,d_{1}\in D_{1}\right\}
. $
\end{proposition}

\begin{proof}
As $\phi :S\rightarrow R$ is an epimorphism, $R\cong S/K.$ Since $K$ acts trivially on $%
D_{2}/KD_{2},D_{1}/KD_{1},$ $R\cong S/K$ acts
on $D_{2}/KD_{2}$ by $r\cdot \left( d_{2}+KD_{2}\right) =\left( s+K\right)
\cdot \left( d_{2}+KD_{2}\right) =s\cdot d_{2}+KD_{2}$ and $R\cong S/K$ acts
on $D_{1}/KD_{1}$ by $r\cdot \left( d_{1}+KD_{1}\right) =\left( s+K\right)
\cdot \left( d_{1}+KD_{1}\right) =s\cdot d_{1}+KD_{1}.$
\begin{equation*}
D_{2}/KD_{2}\overset{\partial _{2\ast }}{\rightarrow }D_{1}/KD_{1}\overset{%
\partial _{1\ast }}{\rightarrow }R
\end{equation*}%
is a $2$-crossed module where
\begin{equation*}
\partial _{2\ast }(d_{2}+KD_{2})=\partial _{2}(d_{2})+KD_{1},\partial
_{1\ast }(d_{1}+KD_{1})=\partial _{1}(d_{1})+K
\end{equation*}%
and $D_{1}/KD_{1}$ acts on $D_{2}/KD_{2}$ by $(d_{1}+KD_{1})\cdot
(d_{2}+KD_{2})=d_{1}\cdot d_{2}+KD_{2}.$ As
\begin{equation*}
\partial _{1_{\ast }}\left( \partial _{2_{\ast }}(d_{2}+KD_{2})\right)
=\partial _{1_{\ast }}(\partial _{2}(d_{2})+KD_{1})=\partial _{1}\left(
\partial _{2}(d_{2})\right) +K=0+K\cong 0_{R},
\end{equation*}%
$D_{2}/KD_{2}\overset{\partial _{2\ast }}{\rightarrow }D_{1}/KD_{1}\overset{%
\partial _{1\ast }}{\rightarrow }R$ is a complex of $k$-algebras.

The Peiffer lifting%
\begin{equation*}
\left\{ -,-\right\} :D_{1}/KD_{1}\times D_{1}/KD_{1}\rightarrow
D_{2}/KD_{2}
\end{equation*}%
is given by $\left\{ d_{1}+KD_{1},d_{1}^{\prime }+KD_{1}\right\} =\left\{
d_{1},d_{1}^{\prime }\right\} +KD_{2}.$

\textbf{PL1:}
\\
\\
$
\begin{array}{rl}
& \partial _{2_{\ast }}\left\{ d_{1}+KD_{1},d_{1}^{\prime }+KD_{1}\right\}
\\
= & \partial _{2_{\ast }}\left( \left\{ d_{1},d_{1}^{\prime }\right\}
+KD_{2}\right)  \\
= & \partial _{2}(\left\{ d_{1},d_{1}^{\prime }\right\} )+KD_{1} \\
= & \left( d_{1}d_{1}^{\prime }-d_{1}\cdot \partial _{1}(d_{1}^{\prime
})\right) +KD_{1} \\
= & \left( d_{1}d_{1}^{\prime }+KD_{1}\right) -\left( d_{1}\cdot \partial
_{1}(d_{1}^{\prime })+KD_{1}\right)  \\
= & \left( d_{1}d_{1}^{\prime }+KD_{1}\right) -\left( d_{1}+KD_{1}\right)
\cdot \left( \partial _{1}(d_{1}^{\prime })+K\right)  \\
= & \left( d_{1}+KD_{1}\right) \left( d_{1}^{\prime }+KD_{1}\right) -\left(
d_{1}+KD_{1}\right) \cdot \partial _{1_{\ast }}(d_{1}^{\prime }+KD_{1}).%
\end{array}%
$
\\

\textbf{PL2:}%
\begin{equation*}
\begin{array}{ccl}
\left\{ \partial _{2_\ast }\left( d_{2}+KD_{2}\right) ,\partial
_{2_\ast }\left( d_{2}^{\prime }+KD_{2}\right) \right\} & = &
\left\{
\partial
_{2}(d_{2})+KD_{1},\partial _{2}(d_{2}^{\prime })+KD_{1}\right\} \\
& = & \left\{ \partial _{2}(d_{2}),\partial _{2}(d_{2}^{\prime })\right\}
+KD_{2} \\
& = & d_{2}d_{2}^{\prime }+KD_{2} \\
& = & \left( d_{2}+KD_{2}\right) \left( d_{2}^{\prime }+KD_{2}\right) .%
\end{array}%
\end{equation*}%
The rest of axioms of $2$-crossed module is given by in appendix.
\begin{equation*}
(\phi ^{\prime \prime },\phi ^{\prime },\phi ):\left\{
D_{2},D_{1},S,\partial _{2},\partial _{1}\right\} \rightarrow
\left\{ D_{2}/KD_{2},D_{1}/KD_{1},R,\partial _{2_\ast },\partial
_{1_{_\ast }}\right\}
\end{equation*}%
or diagrammatically,%
\begin{equation*}
\xymatrix { D_2 \ar[r]^{\phi^{\prime\prime}} \ar[d]_{\partial_{2}}&
D_2/KD_2 \ar[d]^{\partial _{2_{\ast }}} & \\ \ D_1
\ar[r]^{\phi^{\prime}}
\ar[d]_{\partial_{1}}& D_1/ KD_1 \ar[d]^{\partial _{1_{\ast }}} & \\
\ S \ar[r]^\phi & R & }
\end{equation*}%
is a morphism of $2$-crossed modules. (See appendix.)

Suppose that%
\begin{equation*}
(f_{2},f_{1},\phi ):\left\{ D_{2},D_{1},S,\partial _{2},\partial
_{1}\right\} \rightarrow \left\{ B_{2},B_{1},R,\partial _{2}^{\prime
},\partial _{1}^{\prime }\right\}
\end{equation*}%
is any $2$-crossed module morphism. Then we will show that there is a unique
$2$-crossed module morphism%
\begin{equation*}
(f_{2_\ast },f_{1_{\ast }},id_{R}):\left\{
D_{2}/KD_{2},D_{1}/KD_{1},R,\partial _{2_\ast },\partial _{1_\ast
}\right\} \rightarrow \left\{ B_{2},B_{1},R,\partial _{2}^{\prime
},\partial _{1}^{\prime }\right\}
\end{equation*}%
\begin{equation*}
\xymatrix { D_2/KD_2 \ar[r]^{\partial _{2_{\ast }}}
\ar[d]_{f_2{_\ast}} & D_1/KD_1 \ar[r]^{\partial _{1_\ast }}
\ar[d]_{f_1{_\ast}} & R \ar@{=}[d]_{id_R} & \\ \ B_2
\ar[r]_{\partial_{2}^{\prime}} & B_1 \ar[r]_{\partial_{1}^{\prime}}
& R & }
\end{equation*}%
where $f_{2_\ast }(d_{2}+KD_{2})=f_{2}(d_{2})$ and $f_{1_\ast
}(d_{1}+KD_{1})=f_{1}(d_{1})$. Since
\begin{equation*}
f_{2}(k\cdot d_{2})=\phi (k)\cdot f_{2}(d_{2})=0_{R}\cdot
f_{2}(d_{2})=0_{B_{2}}
\end{equation*}%
and\ similarly $f_{1}(d_{1}+KD_{1})=0_{B_{1}},$ $f_{2}(KD_{2})=0_{B_{2}}$%
\ and $f_{1}(KD_{1})=0_{B_{1}},\ f_{2_\ast }$ and $f_{1_\ast }$ are well
defined.

First let us check that $(f_{2_\ast },f_{1_\ast },id_{R})$ is a
$2$-crossed module morphism. For $d_{2}+KD_{2}\in
D_{2}/KD_{2},d_{1}+KD_{1}\in
D_{1}/KD_{1}$ and $r\in R,$%
\begin{equation*}
\begin{array}{rl}
f_{2_\ast }\left( r\cdot (d_{2}+KD_{2})\right) & =f_{2_\ast }\left(
\left( s+K\right) \cdot (d_{2}+KD_{2})\right) \\
& =f_{2_\ast }((s\cdot d_{2})+KD_{2}) \\
& =f_{2}(s\cdot d_{2}) \\
& =\phi (s)\cdot f_{2}(d_{2}) \\
& =s\cdot f_{2}(d_{2}) \\
& =\left( s+K\right) \cdot f_{2_\ast }(d_{2}+KD_{2}) \\
& =r\cdot f_{2_{\ast }}(d_{2}+KD_{2})%
\end{array}%
\end{equation*}%
similarly $f_{1_\ast }\left( r\cdot (d_{1}+KD_{1})\right) =r\cdot
f_{1_\ast }(d_{1}+KD_{1})$,
\begin{equation*}
\begin{array}{rl}
f_{1_\ast }\partial _{2_\ast }(d_{2}+KD_{2}) & =f_{1_\ast }(\partial
_{2}\left( d_{2}\right)+KD_{2}) \\
& =f_{1}\left( \partial _{2}\left( d_{2}\right) \right) \\
& =\partial _{2}^{\prime }\left( f_{2}\left( d_{2}\right) \right) \\
& =\partial _{2}^{\prime }f_{2_\ast }\left( d_{2}+KD_{2}\right)%
\end{array}%
\end{equation*}%
similarly $\partial _{1}^{\prime }f_{1_\ast }=id_{R}\partial
_{1_{\ast }}$
and%
\begin{equation*}
\begin{array}{cl}
f_{2_{\ast }}\{-,-\}\left( d_{1}+KD_{1},d_{1}^{\prime }+KD_{1}\right) &
=f_{2_{\ast }}\{d_{1}+KD_{1},d_{1}^{\prime }+KD_{1}\} \\
& =f_{2_{\ast }}\left( \{d_{1},d_{1}^{\prime }\}+KD_{2}\right) \\
& =f_{2}\{d_{1},d_{1}^{\prime }\} \\
& =f_{2}\{-,-\}(d_{1},d_{1}^{\prime }) \\
& =\{-,-\}\left( f_{1}\times f_{1}\right) (d_{1},d_{1}^{\prime }) \\
& =\left\{ f_{1}(d_{1}),f_{1}(d_{1}^{\prime })\right\} \\
& =\left\{ f_{1_{\ast }}(d_{1}+KD_{1}),f_{1_{\ast }}(d_{1}+KD_{1})\right\}
\\
& =\{-,-\}\left( f_{1_{\ast }}\times f_{1_{\ast }}\right) \left(
d_{1}+KD_{1},d_{1}^{\prime }+KD_{1}\right) .%
\end{array}%
\end{equation*}%
So $(f_{2_\ast },f_{1_\ast },id_{R})$ is a morphism of $2$-crossed
modules.
Furthermore; following equations are verified.%
\begin{equation*}
f_{2_\ast }\phi ^{\prime \prime }=f_{2}\text{ and }f_{1_\ast }\phi
^{\prime }=f_{1}.
\end{equation*}%
So given any morphism of $2$-crossed modules%
\begin{equation*}
(f_{2},f_{1},\phi ):\left\{ D_{2},D_{1},S,\partial _{2},\partial
_{1}\right\} \rightarrow \left\{ B_{2},B_{1},R,\partial _{2}^{\prime
},\partial _{1}^{\prime }\right\}
\end{equation*}%
then there is a unique $(f_{2_\ast },f_{1_\ast },id_{R})$
$2$-crossed module morphism that commutes the following diagram:
\begin{equation*}
\xymatrix { (D_2, D_1, S, \partial_2, \partial_1)
\ar[drr]^{(\phi^{\prime \prime},\phi^{\prime},\phi)} \ar[d]_{(f_2,f_1,\phi)} & & & \\
\ (B_2, B_1, R, \partial _{2}^{\prime }, \partial _{1}^{\prime }) &
& (D_2/KD_2, D_1/KD_1, R,
\partial_{2_{\ast}}, \partial_{1_\ast}) \ar@{-->}[ll]^-{(f_{2_\ast},f_{1_\ast},id_R)}& }
\end{equation*}%
or more simply as%
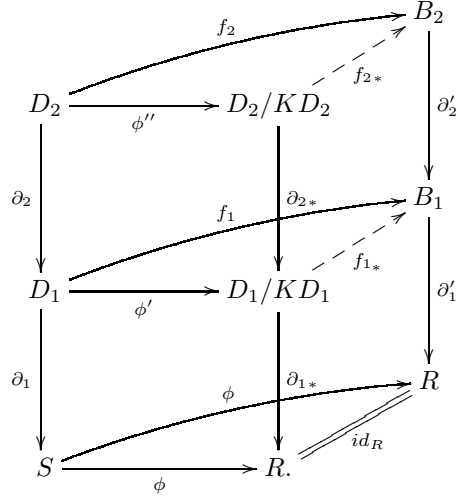
\[
\xymatrix @R=20pt@C=25pt { & & & B_2  \ar[dd]^{\partial _{2}^{\prime
}} & \\
\ D_2  \ar[dd]_{\partial_2} \ar[rr]_-{\phi^{\prime \prime}} \ar@/^/[urrr]^{f_2} &  & D_2/KD_2 \ar@{-->}[ur]_{f_2{_\ast}} \ar[dd]^{\partial_2{_\ast}} &  & \\
\ & & & B_1  \ar[dd]^{\partial _{1}^{\prime
}} & \\
\ D_1 \ar[dd]_{\partial_1} \ar[rr]_-{\phi^\prime} \ar@/^/[urrr]^{f_1} &  & D_1/KD_1 \ar@{-->}[ur]_{f_{1_\ast}} \ar[dd]^{\partial_1{_\ast}} & & \\
\  & & & R  & \\
\  S \ar@/^/[urrr]^\phi \ar[rr]_\phi &  & R. \ar@{=}[ur]_{id_R}& & }
\]
\end{proof}

\section{Fibrations and Cofibrations of Categories}\label{s1}
The notion of fibration of categories is intended to give a general
background to constructions analogous to pullback by a morphism. It
seems to be a very useful notion for dealing with hierarchical
structures. A functor which forgets the top level of structure is
often usefully seen as a fibration or cofibration of categories.

We rewrite from \cite{[bs]} and \cite{Vistoli} the definition of fibration and cofibration of categories and some propositions.
\begin{definition} Let $ \Phi :\mathbf{X}\to \mathbf{B} $ be a functor. A morphism $ \varphi :Y\to
X $ in $ \mathbf{X} $ over $ u := \Phi (\varphi) $ is called
cartesian if and only if for all $ v: K\to J $ in $ \mathbf{B} $ and
$ \theta:Z\to X $ with $ \Phi(\theta)=uv $ there is a unique
morphism $ \psi: Z\to Y $ with $ \Phi(\psi)=v $ and $
\theta=\varphi\psi . $

This is illustrated by the following diagram:
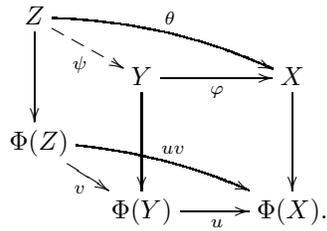
\[
\xymatrix @R=10pt@C=10pt{ Z \ar[dd] \ar@{-->}[dr]_\psi \ar@/^/[drrr]^\theta & & & & \\
\ & Y \ar[dd] \ar[rr]_\varphi & & X \ar[dd] & \\
\ \Phi(Z) \ar[dr]_v \ar@/^/[drrr]^{uv} & & & &\\
\ & \Phi(Y) \ar[rr]_u & & \Phi(X). & }
\]

If $Y\rightarrow X$ is a cartesian arrow of \textbf{X} mapping to an
arrow $\Phi \left( Y\right) \rightarrow \Phi \left( X\right) $ of
\textbf{B}, we also say that Y is a pullback of X to $\Phi \left(
Y\right) $

A morphism $ \alpha :Z\to Y  $ is called vertical (with respect to $
\Phi $) if and only if $ \Phi(\alpha) $  is an identity morphism in
$ \mathbf{B}. $ In particular, for $ I\in \mathbf{B}$ we write
$\mathbf{X}$/I , called the fibre over $ I $, for the subcategory of
$ \mathbf{X} $ consisting of those morphisms $ \alpha $ with $
\Phi(\alpha)=id_I . $
\end{definition}

\begin{definition}
The functor $ \Phi :\mathbf{X}\to \mathbf{B} $ is a fibration or
category fibred over $ \mathbf{B} $ if and only if for all $ u: J\to
I $ in $ \mathbf{B} $ and $ X\in\mathbf{X}$/I there is a cartesian
morphism $\varphi: Y\to X $ over $u$. Such a $ \varphi $ is called a
cartesian lifting of  $ X $ along $ u. $
\end{definition}

In other words, in a category fibred over $\mathbf{B}$, $ \Phi
:\mathbf{X}\to \mathbf{B} $, we can pull back objects of
$\mathbf{X}$ along any arrow of $\mathbf{B}$.

Notice that cartesian liftings of X $\in X/I$ along $u:J\rightarrow I$ are
unique up to vertical isomorphism:
given two pullbacks $\varphi :Y\rightarrow X$ and $\bar{\varphi}:\bar{Y}%
\rightarrow X$ of $X$ to I,
 the unique arrow $\theta :\bar{Y}\rightarrow Y$ that fits into the diagram%
\[
\xymatrix @R=10pt@C=10pt{ \bar{Y} \ar[dd] \ar@{-->}[dr] \ar@/^/[drrr] & & & & \\
\ & Y \ar[dd] \ar[rr] & & X \ar[dd] & \\
\ J \ar@{=}[dr] \ar@/^/[drrr]& & & &\\
\ & J \ar[rr] & & I & }
\]

\noindent
is an isomorphism; the inverse is the arrow $Y\rightarrow \bar{Y}$ obtained
by exchanging $Y$ and $\bar{Y}$ in the diagram above.\
In other words, a pullback is unique, up to a unique isomorphism.

The following results in the case of crossed modules of groupoids
and commutative algebras have appeared in \cite{[bs]} and \cite{ag},
respectively.

\begin{proposition} \label{g2}
The forgetful functor $p:\mathsf{X}_{2}\mathsf{Mod}\rightarrow \mathsf{k}$%
\textsf{-Alg } which sends $\{ C_{2},C_{1},$ \\ $ R,\partial
_{2},\partial _{1}\}\mapsto R$ is fibred.
\end{proposition}

\begin{proof}
It is enough to consider the pullback construction from proposition \ref{ug2} to prove
that $p$ is a fibred. Thus the morphism $(id_{C_{2}},\phi ^{\prime },\phi
):\left\{ C_{2},\phi ^{\ast }(C_{1}),S,\partial _{2}^{\ast
},\partial _{1}^{\ast }\right\}
$ \\ $ \rightarrow \left\{ C_{2},C_{1},R,\partial _{2},\partial _{1}\right\} $ of $%
\mathsf{X}_{2}\mathsf{Mod}$ is cartesian. Because for any morphism
\[
(f_{2},f_{1},\phi ):\left\{ B_{2},B_{1},S,\partial _{2}^{\prime
},\partial _{1}^{\prime }\right\} \rightarrow \left\{
C_{2},C_{1},R,\partial _{2},\partial _{1}\right\}
\]
in $\mathsf{X}_{2}\mathsf{Mod}$ and any morphism
\[
id_{S}:p\left( \left\{ B_{2},B_{1},S,\partial _{2}^{\prime
},\partial _{1}^{\prime }\right\} \right) \rightarrow p\left(
\left\{ C_{2},\phi ^{\ast }(C_{1}),S,\partial _{2}^{\ast },\partial
_{1}^{\ast }\right\} \right)
\]
\noindent
in \textsf{k-Alg} with $p(id_{C_{2}},\phi ^{\prime},\phi)\circ id_{S}=p(f_{2},f_{1},\phi ),$ there exists a unique arrow $%
(f_{2}^{\ast },f_{1}^{\ast },id_{S})$ with $p(f_{2}^{\ast
},f_{1}^{\ast },id_{S})=id_{S}$ and $(id_{C_{2}},\phi ^{\prime},\phi
)\circ (f_{2}^{\ast},f_{1}^{\ast },id_{S})=(f_{2},f_{1},\phi )$ as
in the commutative diagram
\[
 \xymatrix @R=15pt@C=20pt{ \left\{ B_{2},B_{1},S,\partial _{2}^{\prime
},\partial
_{1}^{\prime }\right\} \ar[dd]_p \ar@{-->}[dr]_-{(f_{2}^{\ast },f_{1}^{\ast },id_{S})} \ar@/^/[drrr]^-{(f_{2},f_{1},\phi )} & & &  \\
\ & \left\{ C_{2},\phi ^{\ast }(C_{1}),S,\partial _{2}^{\ast
},\partial _{1}^{\ast }\right\}
 \ar[dd]_p \ar[rr]_-{(id_{C_{2}},\phi ^{\prime },\phi )} & & \left\{ C_{2},C_{1},R,\partial
_{2},\partial _{1}\right\} \ar[dd]_p & \\
\ S \ar@{=}[dr]_{id_S} \ar@/^/[drrr]^\phi & & &\\
\ & S \ar[rr]_\phi  &  & R. }
\]
\end{proof}

In considering the functor $p:\textsf{X}_{2}\textsf{Mod} \rightarrow %
\textsf{k-Alg}$ as a fibration, if we fix the monomorphism $\phi
:S\rightarrow R$ \ in \textsf{k-Alg},  the cartesian lifting $\left\{ C_{2},\phi ^{\ast }\left( C_{1}\right) ,R,\partial _{2},\partial
_{1}\right\} \\ \rightarrow \left\{ C_{2},C_{1},R,\partial
_{2},\partial _{1}\right\} $ \ of \ $\left\{ C_{2},C_{1},R,\partial
_{2},\partial _{1}\right\} $ along $\phi $ \ for \ $\left\{
C_{2},C_{1},R,\partial _{2},\partial _{1}\right\} \in
\mathsf{X}_{2}\mathsf{Mod}/R$, \ gives a so-called
reindexing functor%
\[
\phi ^{\ast }:\mathsf{X}_{2}\mathsf{Mod/}R\rightarrow \mathsf{X}_{2}\mathsf{%
Mod/}S
\]%
given in section 2 and defined as objects by $\left\{
C_{2},C_{1},R,\partial _{2},\partial _{1}\right\} \mapsto \{
C_{2},\phi^\ast(C_{1}),R,$ \\ $\partial _{2},\partial _{1}\} $ and
the image of a morphism $\phi ^{\ast }(\alpha ,\beta ,id_{R})=\left(
\alpha ,\phi ^{\ast }(\beta ),id_{S}\right) $ the unique arrow
commuting the quadrangles in the following
diagram:%
\[
\xymatrix @R=10pt@C=10pt{ C_2 \ar[dd] \ar@{-->}[dr]^\alpha \ar@{=}[rr] &  & C_2 \ar[dd] \ar[drr]^\alpha & & & \\
\ & {C_2}^\prime \ar[dd] \ar@{=}[rrr] &  & & {C_2}^\prime \ar[dd]  &  \\
\ \phi^\ast(C_1) \ar[dd] \ar@{-->}[dr]^{\phi^\ast({\beta})} \ar[rr] &  & C_1 \ar[dd] \ar[drr]^\beta & & & \\
\ & \phi^\ast({C_1}^\prime) \ar[dd]  \ar[rrr] & & & {C_1}^\prime \ar[dd] &  \\
\ S \ar@{-->}[dr]^{id_S} \ar[rr] &  & R \ar@{=}[drr]^{id_R} & & & \\
\ & S \ar[rrr]_\phi &  & & R. &  }
\]

We can use this reindexing functor to get an adjoint situation for each
monomorphism $\phi :S\rightarrow R$ in \textsf{k}-\textsf{Alg}.

\begin{proposition}
Let  $p:$ \textsf{X}$_{2}$\textsf{Mod}$\rightarrow $\textsf{k}-\textsf{Alg }%
be the forgetful functor with $p \{C_{2},C_{1},R,$ \\  $\partial _{2}, \partial _{1}\} \mapsto R$, monomorphism $\phi :S\rightarrow R$ be
in \textsf{k}-\textsf{Alg}, and $\phi ^{\ast }:\mathsf{X}_{2}\mathsf{Mod/}%
R\rightarrow \mathsf{X}_{2}\mathsf{Mod/}S$ be reindexing functor, pullback.
Then there is a bijection
\[
\mathsf{X}_{2}\mathsf{Mod/}S\left( \mathcal{B},\phi ^{\ast }\left( \mathcal{C%
}\right) \right) \cong \mathsf{X}_{2}\mathsf{Mod/}\phi \left( \mathcal{B},%
\mathcal{C}\right)
\]%
natural in $\mathcal{B}\in \mathsf{X}_{2}\mathsf{Mod/}S,\ \mathcal{C}\in \mathsf{X}_{2}\mathsf{Mod/}R$ where $\mathcal{B}=\left\{ B_{2},B_{1},S,\partial _{2}^{\prime },\partial_{1}^{\prime }\right\},\ \mathcal{C}=\left\{ C_{2},C_{1},R,\partial
_{2},\partial _{1}\right\} $ and $\mathsf{X}%
_{2}\mathsf{Mod/}\phi \left( \mathcal{B},\mathcal{C}\right) $ consists of those
morphisms $\left( f_{2},f_{1},\phi \right) \in \mathsf{X}_{2}\mathsf{Mod}%
\left( \mathcal{B},\mathcal{C}\right) $ with $p\left( f_{2},f_{1},\phi
\right) =\phi .$
\end{proposition}

\begin{proof}
Since $(id_{C_{2}},\phi ^{\prime },\phi
):\left\{ C_{2},\phi ^{\ast }(C_{1}),S,\partial _{2}^{\ast
},\partial _{1}^{\ast }\right\}
\rightarrow \left\{ C_{2},C_{1},R,\partial _{2},\partial _{1}\right\} $ is a cartesian
morphism over $\phi $ as may be seen with the proof of proposition \ref{g2}, there is a
unique morphism
\[
\Psi :\left\{ B_{2},B_{1},S,\partial _{2}^{\prime
},\partial _{1}^{\prime }\right\} \rightarrow \left\{ C_{2},\phi ^{\ast
}\left( C_{1}\right) ,S,\partial _{2}^{\ast },\partial _{1}^{\ast }\right\}
\]
with $p\Psi =id_{S}$ for $\left( f_{2},f_{1},\phi \right) :\left\{
B_{2},B_{1},S,\partial _{2}^{\prime },\partial _{1}^{\prime }\right\}
\rightarrow \left\{ C_{2},C_{1},R,\partial _{2},\partial _{1}\right\} $ with
$p(\left( f_{2},f_{1},\phi \right) )=\phi .$

On the other hand; it is clear that there is a morphism
\[
\phi ^{\prime }\circ
\gamma :\left\{ B_{2},B_{1},S,\partial _{2}^{\prime },\partial _{1}^{\prime
}\right\} \rightarrow \left\{ C_{2},C_{1},R,\partial _{2},\partial
_{1}\right\}
\]
for $\gamma :\left\{ B_{2},B_{1},S,\partial _{2}^{\prime
},\partial _{1}^{\prime }\right\} \rightarrow \left\{ C_{2},\phi ^{\ast
}\left( C_{1}\right) ,S,\partial _{2}^{\ast },\partial _{1}^{\ast }\right\} .
$
\end{proof}

For composable monomorphism $\phi :S\rightarrow R$ $\ $and $\phi ^{\prime
}:R\rightarrow T$, there is a natural equivalence
\[
c_{\phi ^{\prime },\phi }:\phi ^{\ast }\phi ^{\prime \ast }\cong \left( \phi
^{\prime }\phi \right) ^{\ast }
\]%
but not equality as shown in the following diagram in which the morphisms $\left(
id_{C_{2}},h,\phi ^{\prime }\phi \right) ,\left( id_{C_{2}},g,\phi \right) $
and $\left( id_{C_{2}},g^{\prime },\phi ^{\prime }\right) $ are cartesian
and so the composition $\left( id_{C_{2}},g,\phi \right) \circ \left(
id_{C_{2}},g^{\prime },\phi ^{\prime }\right) $ is cartesian and $%
\left( id_{C_{2}},k,id_{S}\right) $ is the unique vertical morphism from
proposition \ref{g2} making the diagram commute:
\[
\xymatrix @R=20pt@C=25pt{ C_2 \ar@/^/[drrr]^{id_{C_2}} \ar[dd]  & & & & \\
\ & C_2 \ar@{=}[r]_-{id_{C_2}} \ar[dd] \ar@{-->}[ul]^{id_{C_2}} & C_2 \ar@{=}[r]_-{id_{C_2}} \ar[dd] & C_2 \ar[dd] & \\
\ ({\phi^\prime}\phi)^\ast (C_1) \ar@/^/[drrr]^h \ar[dd]  & & & & \\
\ & \phi^\ast{{\phi^\prime}^\ast(C_1)} \ar[r]_-g \ar[dd] \ar@{-->}[ul]^k & {{\phi^\prime}^\ast(C_1)} \ar[r]_-{g^{\prime}} \ar[dd] & C_1 \ar[dd] & \\
\ S \ar@/^/[drrr]^{{\phi^\prime}\phi}  & & & & \\
\ & S \ar[r]_\phi \ar@{-->}[ul]^{id_S}& R \ar[r]_{\phi^\prime} & T. & }
\]

\begin{definition} Let $ \Phi :\mathbf{X}\to \mathbf{B} $ be a functor. A morphism $ \psi :Z\to
Y $ in $ \mathbf{X} $ over $ v := \Phi (\psi) $ is called
cocartesian if and only if for all $ u: J\to I $ in $ \mathbf{B} $
and $ \theta:Z\to X $ with $ \Phi(\theta)=uv $ there is a unique
morphism $ \varphi: Y\to X $ with $ \Phi(\varphi)=u $ and $
\theta=\varphi\psi . $ This is illustrated by the following diagram:
\[
\xymatrix @R=10pt@C=10pt{ & & & X  \ar[dd] & \\
\ Z \ar[dd] \ar[rr]^\psi \ar@/^/[urrr]^\theta & & Y \ar[dd] \ar@{-->}[ur]^\varphi & & \\
\ & & & \Phi(X) & \\
\ \Phi(Z) \ar[rr]_v \ar@/^/[urrr]^{uv} & & \Phi(Y). \ar[ur]_u& & }
\]
The functor $ \Phi :\mathbf{X}\to \mathbf{B} $ is a cofibration or
category cofibred over $ \mathbf{B} $ if and only if for all $ v:
K\to J $ in $ \mathbf{B} $ and $ Z\in\mathbf{X}$/K there is a
cocartesian morphism $ \psi: Z\to Z^\prime $ over $v$. Such a
$ \psi $ is called a cocartesian lifting of  $ Z $ along $ v. $

The cocartesian liftings of $ Z\in\mathbf{X}$/K along $ v: K\to J $
are also unique up to vertical isomorphism.
\end{definition}

It is interesting to get a characterisation of the cofibration property for a functor that already is a fibration. The following is a useful weakening of the condition for cocartesian in the case of a fibration of categories.
\begin{proposition}\label{ug5}
Let $ \Phi :\mathbf{X}\to \mathbf{B} $ be a fibration of categories. Then $ \psi:Z\to Y $ in $\mathbf{X}$ over $ v:K\to J $ in $\mathbf{B}$ is cocartesian if only if for all $\theta ^{\prime }:Z\rightarrow X^{\prime }$ over $v$ there is a unique
morphism $\psi ^{\prime }:Y\rightarrow X^{\prime }$ in \textbf{X}/J with
$\theta ^{\prime }=\psi ^{\prime }\psi .$
\end{proposition}

If we take the fibration $p:\mathsf{X}_{2}\mathsf{Mod}\rightarrow
\mathsf{k}$\textsf{-Alg} and reindexing functor $\phi ^{\ast }$ $: $
\textsf{X}$_{\mathbf{2}}$\textsf{Mod}/$R\rightarrow $\textsf{X}$_{\mathbf{2}}
$\textsf{Mod}/$S$ for monomorphism $\phi :S\rightarrow R$, we get the morphism
\[
\phi_{\left\{ D_{2},D_{1},S,\partial _{2},\partial _{1}\right\} }:\left\{
D_{2},D_{1},S,\partial _{2},\partial _{1}\right\} \rightarrow \phi _{\ast
}\left\{ D_{2},D_{1},S,\partial _{2},\partial _{1}\right\}
\]
which is cocartesian over $\phi $, for all $\left\{ D_{2},D_{1},S,\partial
_{2},\partial _{1}\right\} \in\ $\textsf{X}$_{\mathbf{2}}$\textsf{Mod}/$S.$

Because there is the functor $\phi _{\ast }:$\textsf{X}$_{\mathbf{2}}$%
\textsf{Mod}/$S\rightarrow $\textsf{X}$_{\mathbf{2}}$\textsf{Mod}/$R$ which
is left adjoint to $\phi ^{\ast }$ as mentioned in theorem \ref{ug6}, and also by proposition \ref{ug5}, the adjointness
gives the bijection required for cocartesian property. Thus  $\phi
_{\left\{ D_{2},D_{1},S,\partial _{2},\partial _{1}\right\} }$ is
cocartesian over $\phi .$

So, by constructing the adjoint  $\phi _{\ast }$ of $\phi ^{\ast }$ for $%
\phi ,$ the fibration $p$ is also a cofibration.

\section{Application: Free 2-Crossed Modules}

The definition of a free 2-crossed module is similar in some ways to the corresponding definition of a free crossed module. We recall the definition of a free crossed module and a free 2-crossed module from \cite{[pa]}.

Let $(C,R,\partial )$ be a pre-crossed module, let $Y$ be a set and let  $%
\vartheta :Y\rightarrow C$ \ be a function, then $(C,R,\partial )$
is said to be a free pre-crossed module with basis $\vartheta $ or,
alternatively, on the function $\partial \vartheta :Y\rightarrow R$
if for any pre-crossed module $(C^{\prime },R,\partial ^{\prime })$
and function $\vartheta ^{\prime }:Y\rightarrow C^{\prime }$ such
that $\partial ^{\prime }\vartheta
^{\prime }=\partial \vartheta $, there is a unique morphism%
\[
\phi :(C,R,\partial )\rightarrow (C^{\prime },R,\partial ^{\prime })
\]%
such that $\phi \vartheta =\vartheta ^{\prime }$.

Let $\left\{ C_{2},C_{1},C_{0},\partial _{2},\partial _{1}\right\} $ be a $2$%
-crossed module, let $Y$ be a set and let $\vartheta :Y\rightarrow C_{2}$ be
a function, then $\left\{ C_{2},C_{1},C_{0},\partial _{2},\partial
_{1}\right\} $ is said to be a free $2$-crossed module with basis $\vartheta
$ or, alternatively, on the function $\partial _{2}\vartheta :Y\rightarrow
C_{1}$, if for any $2$-crossed module $\left\{ C_{2}^{\prime
},C_{1},C_{0},\partial _{2}^{\prime },\partial _{1}\right\} $ and function $%
\vartheta ^{\prime }:Y\rightarrow C_{2}^{\prime }$ such that $\partial
_{2}\vartheta =\partial _{2}^{\prime }\vartheta ^{\prime }$, there is a
unique morphism $\Phi :C_{2}\rightarrow C_{2}^{\prime }$ such that $\partial
_{2}^{\prime }\Phi =\partial _{2}$.%
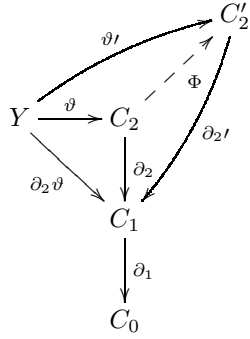
\[
\xymatrix { & & C_{2}^{\prime} \ar@/^/[ddl]^{\partial_{2}\prime} & \\
\ Y \ar[r]^\vartheta  \ar@/^/[urr]^{\vartheta\prime} \ar[dr]_{\partial_{2}\vartheta} & C_{2} \ar[d]^{\partial_{2}} \ar@{-->}[ur]_\Phi & & \\
\  & C_{1} \ar[d]^{\partial_{1}} & & \\
\ & C_{0} & & }
\]
\

The following proposition is an application of induced $2$-crossed modules using the universal properties of free $2$-crossed modules and  of universal morphism of $2$-crossed modules.

\begin{proposition} \label{freeness}
Suppose $\phi :S\rightarrow R$ is a $k$-algebra morphism, then the $2$%
-crossed module $\left\{ C_{2},C_{1},R,\partial _{2},\partial
_{1}\right\} $ is the free $2$-crossed module on $\partial
_{2}\vartheta $ with the morphism $\vartheta :Y\rightarrow C_{2}$ if
and only if the morphism $\left( \vartheta ,\partial _{2}\vartheta
,\phi \right) :\{Y,Y,S,id_{Y},0\}\rightarrow $ $\left\{
C_{2},C_{1},R,\partial _{2},\partial _{1}\right\} $ of $2$-crossed
modules is a universal morphism, i.e. for any $2$-crossed module
morphism
\[
\left( \vartheta ^{\prime },\partial _{2}\vartheta ,\phi \right) :\left\{
Y,Y,S,id_{Y},0\right\} \rightarrow \left\{ C_{2}^{\prime },C_{1},R,\partial
_{2}^{\prime },\partial _{1}\right\}
\]%
there exists a unique $(\Phi, id_{C_{1}}, id_R) $ $2$-crossed module morphism such
that $\Phi \vartheta =\vartheta ^{\prime }$.
\end{proposition}

\begin{proof}
Suppose $\left( \vartheta ,\partial _{2}\vartheta ,\phi \right) $ is a
universal morphism of $2$-crossed modules. Let
\[
\left( \vartheta ^{\prime },\partial _{2}\vartheta ,\phi \right) :\left\{
Y,Y,S,id_{Y},0\right\} \rightarrow \left\{ C_{2}^{\prime },C_{1},R,\partial
_{2}^{\prime },\partial _{1}\right\}
\]%
be a $2$-crossed module morphism. Then there exists a unique
\[
\left( \Phi ,id_{C_{1}},id_{R}\right) :\left\{ C_{2},C_{1},R,\partial
_{2},\partial _{1}\right\} \rightarrow \left\{ C_{2}^{\prime
},C_{1},R,\partial _{2}^{\prime },\partial _{1}\right\}
\]%
$2$-crossed module morphism such that $\Phi \vartheta =\vartheta ^{\prime }$. This description gives the
required free $2$-crossed module $\left\{ C_{2},C_{1},R,\partial
_{2},\partial _{1}\right\} $ on $\partial _{2}\vartheta $.

On the other hand, let $\left\{ C_{2},C_{1},R,\partial _{2},\partial
_{1}\right\} $ be a free $2$-crossed module on $\partial _{2}\vartheta $, $\left\{ C_{2}^{\prime },C_{1},R,\partial _{2}^{\prime },\partial
_{1}\right\} $ be a $2$-crossed module and $\vartheta ^{\prime
}:Y\rightarrow C_{2}^{\prime }$ be an algebra morphism such that $\partial _{2}\vartheta ={\partial _{2}}^\prime {\vartheta}^\prime $. Then by using the
universal property of a free $2$-crossed module, we get a unique
morphism $\Phi:C_2 \rightarrow {C_2}^\prime$ such that $\partial _{2}^{\prime }\Phi=\partial _{2}$. This proves that
\[
\left( \vartheta ,\partial _{2}\vartheta ,\phi \right)
:\{Y,Y,S,id_{Y},0\}\rightarrow \left\{ C_{2},C_{1},R,\partial _{2},\partial
_{1}\right\}
\]%
is a universal morphism of $2$-crossed modules.
\[
\xymatrix { & & & C_{2}^{\prime} \ar@/^/[ddl]^{{\partial_2}^\prime} &  \\
\ Y \ar[rr]^\vartheta \ar@{=}[d]_{id_Y} \ar@/^/[urrr]^{\vartheta^\prime} & & C_{2} \ar[d]^{\partial_2} \ar@{-->}[ur]^\Phi & & \\
\ Y  \ar[rr]^{{\partial_2}\vartheta} \ar[d]_0  & & C_{1}
\ar[d]^{\partial_1}  & &
\\
\  S \ar[rr]^\phi & & R. & &}
\]
\end{proof}

This proposition leads to link free $2$-crossed modules with induced $2$-crossed modules.

Given any $k$-algebra morphism $\phi :S\rightarrow R,$ we note that $Y%
\overset{id_{Y}}{\rightarrow }Y\overset{0}{\rightarrow }S$ is a $2$-crossed
module and form the induced $2$-crossed module $\left\{ \phi _{\ast
}(Y),\phi _{\ast }(Y),R,id_{Y\ast },0_{\ast }\right\} $ as described in  section \ref{induced}.

Proposition \ref{freeness} implies that the free $2$-crossed module on the morphism $%
Y\rightarrow \phi _{\ast }(Y)$ is the induced $2$-crossed module on $\phi $.

Note that the definition of free $2$-crossed modules has been chosen to tie
in with the definition of freeness given in other more general context.

\begin{proposition}
Let $\left\{ C_{2},C_{1},C_{0},\partial _{2},\partial _{1}\right\} $ be a
free $2$-crossed module on $f$ and $f=0$, then $C_{2}$ is a free $C_{1}$%
-module on $f$.
\end{proposition}

\begin{proof}
Given any free $2$-crossed module\ $\left\{ C_{2},C_{1},C_{0},\partial _{2},\partial _{1}\right\} $ on $f$, then $\left( C_{2},C_{1},\partial _{2}\right) $
is a free crossed module on $f$. Since $\left( C_{2},C_{1},\partial
_{2}\right) $ is free on $f$, then there exists a function $\vartheta
:Y\rightarrow C_{2}$ such that $\partial _{2}\vartheta =f=0$. Let $%
C_{2}^{\prime }$ be a $C_{1}$-module and $\vartheta ^{\prime }:Y\rightarrow
C_{2}^{\prime }$ be a function into $C_{2}^{\prime }$. Form the crossed
module $\left( C_{2}^{\prime },C_{1},0\right) $. Since $\partial
_{2}\vartheta =f=0=0\vartheta ^{\prime }$, $\ $thus there is a unique
morphism $\Phi :C_{2}\rightarrow C_{2}^{\prime }$ such that $\Phi \vartheta
=\vartheta ^{\prime }$. Therefore $C_{2}$ is free $C_{1}$-module on $f$.
\end{proof}

\newpage

\section {Appendix}

\textbf{The proof of proposition \ref{ug2}}

\textbf{PL3:}%
\begin{equation*}
\begin{array}{cll}
& \left\{ \left( c_{1},s\right) ,\left( c_{1}^{\prime },s^{\prime }\right)
\left( c_{1}^{\prime \prime },s^{\prime \prime }\right) \right\} &  \\
= & \left\{ \left( c_{1},s\right) ,\left( c_{1}^{\prime }c_{1}^{\prime
\prime },s^{\prime }s^{\prime \prime }\right) \right\} &  \\
= & \left\{ c_{1},c_{1}^{\prime }c_{1}^{\prime \prime }\right\} &  \\
= & \left\{ c_{1}c_{1}^{\prime },c_{1}^{\prime \prime }\right\} +\partial
_{1}\left( c_{1}^{\prime \prime }\right) \cdot \left\{ c_{1},c_{1}^{\prime
}\right\} &  \\
= & \left\{ c_{1}c_{1}^{\prime },c_{1}^{\prime \prime }\right\} +\phi \left(
s^{\prime \prime }\right) \cdot \left\{ c_{1},c_{1}^{\prime }\right\} &  \\
= & \left\{ c_{1}c_{1}^{\prime },c_{1}^{\prime \prime }\right\} +s^{\prime
\prime }\cdot \left\{ c_{1},c_{1}^{\prime }\right\} &  \\
= & \left\{ c_{1}c_{1}^{\prime },c_{1}^{\prime \prime }\right\} +\partial
_{1}^{\ast }\left( c_{1}^{\prime \prime },s^{\prime \prime }\right) \cdot
\left\{ c_{1},c_{1}^{\prime }\right\} &  \\
= & \left\{ \left( c_{1}c_{1}^{\prime },ss^{\prime }\right) ,\left(
c_{1}^{\prime \prime },s^{\prime \prime }\right) \right\} +\partial
_{1}^{\ast }\left( c_{1}^{\prime \prime },s^{\prime \prime }\right) \cdot
\left\{ \left( c_{1},s\right) ,\left( c_{1}^{\prime },s^{\prime }\right)
\right\} &  \\
= & \left\{ \left( c_{1},s\right) \left( c_{1}^{\prime },s^{\prime }\right)
,\left( c_{1}^{\prime \prime },s^{\prime \prime }\right) \right\} +\partial
_{1}^{\ast }\left( c_{1}^{\prime \prime },s^{\prime \prime }\right) \cdot
\left\{ \left( c_{1},s\right) ,\left( c_{1}^{\prime },s^{\prime }\right)
\right\} . &
\end{array}%
\end{equation*}

\textbf{PL4:}

\begin{equation*}
\begin{array}{cl}
& \left\{ \left( c_{1},s\right) ,\partial _{2}^{\ast }(c_{2})\right\}
+\left\{ \partial _{2}^{\ast }(c_{2}),\left( c_{1},s\right) \right\} \\
= & \left\{ \left( c_{1},s\right) ,(\partial _{2}\left( c_{2}\right)
,0)\right\} +\left\{ (\partial _{2}\left( c_{2}\right) ,0),\left(
c_{1},s\right) \right\} \\
= & \left\{ c_{1},\partial _{2}\left( c_{2}\right) \right\} +\left\{
\partial _{2}\left( c_{2}\right) ,c_{1}\right\} \\
= & \partial _{1}(c_{1})\cdot c_{2} \\
= & \phi \left( s\right) \cdot c_{2} \\
= & s\cdot c_{2} \\
= & \partial _{1}^{\ast }\left( c_{1},s\right) \cdot c_{2}.%
\end{array}%
\end{equation*}

\textbf{PL5:}%
\begin{equation*}
\begin{array}{rcl}
\left\{ \left( c_{1},s\right) \cdot s^{\prime \prime },\left( c_{1}^{\prime
},s^{\prime }\right) \right\} & = & \left\{ \left( c_{1}\cdot \phi
(s^{\prime \prime }),ss^{\prime \prime }\right) ,\left( c_{1}^{\prime
},s^{\prime }\right) \right\} \\
& = & \left\{ c_{1}\cdot \phi (s^{\prime \prime }),c_{1}^{\prime }\right\}
\\
& = & \left\{ c_{1},c_{1}^{\prime }\right\} \cdot \phi (s^{\prime \prime })
\\
& = & \left\{ c_{1},c_{1}^{\prime }\right\} \cdot s^{\prime \prime } \\
& = & \left\{ \left( c_{1},s\right) ,\left( c_{1}^{\prime },s^{\prime
}\right) \right\} \cdot s^{\prime \prime }.%
\end{array}%
\end{equation*}%
\begin{equation*}
\begin{array}{rcl}
\left\{ \left( c_{1},s\right) ,\left( c_{1}^{\prime },s^{\prime }\right)
\cdot s^{\prime \prime }\right\} & = & \left\{ \left( c_{1},s\right) \left(
c_{1}^{\prime }\cdot \phi (s^{\prime \prime }),s^{\prime }s^{\prime \prime
}\right) \right\} \\
& = & \left\{ c_{1},c_{1}^{\prime }\cdot \phi (s^{\prime \prime })\right\}
\\
& = & \left\{ c_{1},c_{1}^{\prime }\right\} \cdot \phi (s^{\prime \prime })
\\
& = & \left\{ c_{1},c_{1}^{\prime }\right\} \cdot s^{\prime \prime } \\
& = & \left\{ \left( c_{1},s\right) ,\left( c_{1}^{\prime },s^{\prime
}\right) \right\} \cdot s^{\prime \prime }%
\end{array}%
\end{equation*}
for all $(c_{1},s),(c_{1}^{\prime },s^{\prime }),(c_{1}^{\prime \prime },s^{\prime
\prime })\in \phi ^{\ast }(C_{1}),c_{2}\in C_{2}$ and $s^{\prime \prime }\in
S.$

Let us check that $\left( id_{C_{2}},\phi ^{\prime },\phi \right) :\left\{
C_{2},\phi ^{\ast }(C_{1}),S,\partial _{2}^{\ast },\partial _{1}^{\ast
}\right\} \rightarrow \left\{ C_{2},C_{1},R,\partial _{2},\partial
_{1}\right\} $ is a morphism of $2$-crossed modules where $\phi ^{\prime
}\left( c_{1},s\right) =c_{1}$.%
\begin{equation*}
\begin{array}{rcl}
id_{C_{2}}\left( s\cdot c_{2}\right) & = & s\cdot c_{2} \\
& = & \phi \left( s\right) \cdot c_{2} \\
& = & \phi \left( s\right) \cdot id_{C_{2}}(c_{2})%
\end{array}%
\end{equation*}%
and similarly $\phi ^{\prime }(s\cdot \left( c_{1},s^{\prime }\right) )=\phi
\left( s\right) \cdot \phi ^{\prime }(c_{1},s^{\prime })$
\begin{equation*}
\begin{array}{rcl}
\left( \phi ^{\prime }\partial _{2}^{\ast }\right) \left( c_{2}\right) & = &
\phi ^{\prime }\left( \partial _{2}\left( c_{2}\right) ,0\right) \\
& = & \partial _{2}\left( c_{2}\right) \\
& = & \partial _{2}id_{C_{2}}\left( c_{2}\right)%
\end{array}%
\end{equation*}%
and $\phi \partial _{1}^{\ast }=\partial _{1}\phi ^{\prime }$ for all $(c_{1},s^{\prime })\in \phi ^{\ast }(C_{1}),c_{2}\in C_{2}$ and $s\in S.$
\[
\begin{array}{rcl}
id_{C_{2}}\left\{ -,-\right\} \left( \left( c_{1},s\right) ,\left(
c_{1}^{\prime },s^{\prime }\right) \right)  & = & id_{C_{2}}\left( \left\{
\left( c_{1},s\right) ,\left( c_{1}^{\prime },s^{\prime }\right) \right\}
\right)  \\
& = & \left\{ c_{1},c_{1}^{\prime }\right\}  \\
& = & \left\{ -,-\right\} \left( c_{1},c_{1}^{\prime }\right)  \\
& = & \left\{ -,-\right\} \left( \phi ^{\prime }\left( c_{1},s\right) ,\phi
^{\prime }\left( c_{1}^{\prime },s^{\prime }\right) \right)  \\
& = & \{-,-\}\left( \phi ^{\prime }\times \phi ^{\prime }\right) \left(
\left( c_{1},s\right) ,\left( c_{1}^{\prime },s^{\prime }\right) \right)
\end{array}%
\]
for all $\left( c_{1},s\right) ,\left( c_{1}^{\prime },s^{\prime }\right)
\in \phi ^{\ast }(C_{1}).$

\textbf{The proof of proposition \ref{ug3}}

\textbf{PL3:}%
\begin{equation*}
\begin{array}{cl}
& \left\{ \left( d_{1},r\right) ,\left( d_{1}^{\prime },r^{\prime }\right)
\left( d_{1}^{\prime \prime },r^{\prime \prime }\right) \right\} \\
= & \left\{ \left( d_{1},r\right) ,\left( d_{1}^{\prime }d_{1}^{\prime
\prime },r^{\prime }r^{\prime \prime }\right) \right\} \\
= & \left( \left\{ d_{1},d_{1}^{\prime }d_{1}^{\prime \prime }\right\}
,r(r^{\prime }r^{\prime \prime })\right) \\
= & \left( \left\{ d_{1}d_{1}^{\prime },d_{1}^{\prime \prime }\right\}
+\partial _{1}\left( d_{1}^{\prime \prime }\right) \cdot \left\{
d_{1},d_{1}^{\prime }\right\} ,r(r^{\prime }r^{\prime \prime })\right) \\
= & \left( \left\{ d_{1}d_{1}^{\prime },d_{1}^{\prime \prime }\right\}
,r(r^{\prime }r^{\prime \prime }))+(\partial _{1}\left( d_{1}^{\prime \prime
}\right) \cdot \left\{ d_{1},d_{1}^{\prime }\right\} ,r(r^{\prime }r^{\prime
\prime })\right) \\
= & \left( \left\{ d_{1}d_{1}^{\prime },d_{1}^{\prime \prime }\right\}
,rr^{\prime }r^{\prime \prime })+(\left\{ d_{1},d_{1}^{\prime }\right\}
,\phi \left( \partial _{1}\left( d_{1}^{\prime \prime }\right) \right)
rr^{\prime }r^{\prime \prime }\right) \\
= & \left( \left\{ d_{1}d_{1}^{\prime },d_{1}^{\prime \prime }\right\}
,rr^{\prime }r^{\prime \prime })+\left( \phi \left( \partial _{1}\left(
d_{1}^{\prime \prime }\right) \right) r^{\prime \prime }\right) \cdot
(\left\{ d_{1},d_{1}^{\prime }\right\} ,rr^{\prime }\right) \\
= & \left( \left\{ d_{1}d_{1}^{\prime },d_{1}^{\prime \prime }\right\}
,rr^{\prime }r^{\prime \prime })+\partial _{1_\ast }(d_{1}^{\prime \prime
},r^{\prime \prime })\cdot (\left\{ d_{1},d_{1}^{\prime }\right\}
,rr^{\prime }\right) \\
= & \left\{ (d_{1}d_{1}^{\prime },rr^{\prime }),(d_{1}^{\prime \prime
},r^{\prime \prime })\right\} +\partial _{1_\ast }(d_{1}^{\prime \prime
},r^{\prime \prime })\cdot \left\{ (d_{1},r),(d_{1}^{\prime },r^{\prime
})\right\} \\
= & \left\{ (d_{1},r)(d_{1}^{\prime },r^{\prime }),(d_{1}^{\prime \prime
},r^{\prime \prime })\right\} +\partial _{1_\ast }(d_{1}^{\prime \prime
},r^{\prime \prime })\cdot \left\{ (d_{1},r),(d_{1}^{\prime },r^{\prime
})\right\} .%
\end{array}%
\end{equation*}

\textbf{PL4:}%
\begin{equation*}
\begin{array}{cl}
& \left\{ \left( d_{1},r\right) ,\partial _{2_{\ast }}(d_{2},r^{\prime
})\right\} +\left\{ \partial _{2_{\ast }}(d_{2},r^{\prime }),\left(
d_{1},r\right) \right\} \\
= & \left\{ \left( d_{1},r\right) ,(\partial _{2}\left( d_{2}\right)
,r^{\prime })\right\} +\left\{ (\partial _{2}\left( d_{2}\right) ,r^{\prime
}),\left( d_{1},r\right) \right\} \\
= & \left( \left\{ d_{1},\partial _{2}\left( d_{2}\right) \right\}
,rr^{\prime }\right) +\left( \left\{ \partial _{2}\left( d_{2}\right)
,d_{1}\right\} ,r^{\prime }r\right) \\
= & \left( \left\{ d_{1},\partial _{2}\left( d_{2}\right) \right\} +\left\{
\partial _{2}\left( d_{2}\right) ,d_{1}\right\} ,rr^{\prime }\right) \\
= & \left( \partial _{1}(d_{1})\cdot d_{2},rr^{\prime }\right) \\
= & \left( d_{2},\phi \left( \partial _{1}(d_{1})\right) rr^{\prime }\right)
\\
= & \left( \phi \left( \partial _{1}(d_{1})\right) r\right) \cdot
(d_{2},r^{\prime }) \\
= & \partial _{1_\ast }(d_{1},r)\cdot (d_{2},r^{\prime }).%
\end{array}%
\end{equation*}

\textbf{PL5:}%
\begin{equation*}
\begin{array}{ccl}
\left\{ \left( d_{1},r\right) \cdot r^{\prime \prime },\left( d_{1}^{\prime
},r^{\prime }\right) \right\} & = & \left\{ \left( d_{1},rr^{\prime \prime
}\right) ,\left( d_{1}^{\prime },r^{\prime }\right) \right\} \\
& = & \left( \left\{ d_{1},d_{1}^{\prime }\right\} ,rr^{\prime \prime
}r^{\prime }\right) \\
& = & \left( \left\{ d_{1},d_{1}^{\prime }\right\} ,rr^{\prime }\right)
\cdot r^{\prime \prime } \\
& = & \left\{ \left( d_{1},r\right) ,\left( d_{1}^{\prime },r^{\prime
}\right) \right\} \cdot r^{\prime \prime }.%
\end{array}%
\end{equation*}%
\begin{equation*}
\begin{array}{ccl}
\left\{ \left( d_{1},r\right) ,\left( d_{1}^{\prime },r^{\prime }\right)
\cdot r^{\prime \prime }\right\} & = & \left\{ (d_{1},r),(d_{1}^{\prime
},r^{\prime }r^{\prime \prime })\right\} \\
& = & \left( \left\{ d_{1},d_{1}^{\prime }\right\} ,rr^{\prime }r^{\prime
\prime }\right) \\
& = & (\left\{ d_{1},d_{1}^{\prime }\right\} ,rr^{\prime })\cdot r^{\prime
\prime } \\
& = & \left\{ \left( d_{1},r\right) ,\left( d_{1}^{\prime },r^{\prime
}\right) \right\} \cdot r^{\prime \prime }%
\end{array}%
\end{equation*}
for all $(d_{1},r),(d_{1}^{\prime },r^{\prime }),(d_{1}^{\prime \prime
},r^{\prime \prime })\in \phi _{\ast }(D_{1}),(d_{2},r^{\prime })\in \phi
_{\ast }(D_{2})$ and $r^{\prime \prime }\in R.$

\newpage
\textbf{The proof of proposition \ref{ug4}}

\textbf{PL3:}%
\begin{equation*}
\begin{array}{cl}
& \left\{ d_{1}+KD_{1},\left( d_{1}^{\prime }+KD_{1}\right) \left(
d_{1}^{\prime \prime }+KD_{1}\right) \right\} \\
= & \left\{ d_{1}+KD_{1},\left( d_{1}^{\prime }d_{1}^{\prime \prime
}+KD_{1}\right) \right\} \\
= & \left\{ d_{1},d_{1}^{\prime }d_{1}^{\prime \prime }\right\} +KD_{2} \\
= & \left( \left\{ d_{1}d_{1}^{\prime },d_{1}^{\prime \prime }\right\}
+\partial _{1}(d_{1}^{\prime \prime })\cdot \left\{ d_{1},d_{1}^{\prime
}\right\} \right) +KD_{2} \\
= & \left( \left\{ d_{1}d_{1}^{\prime },d_{1}^{\prime \prime }\right\}
+KD_{2}\right) +\left( \partial _{1}(d_{1}^{\prime \prime })\cdot \left\{
d_{1},d_{1}^{\prime }\right\} +KD_{2}\right) \\
= & \left( \left\{ d_{1}d_{1}^{\prime },d_{1}^{\prime \prime }\right\}
+KD_{2}\right) +\left( \partial _{1}(d_{1}^{\prime \prime })+K)\cdot \left\{
d_{1},d_{1}^{\prime }\right\} +KD_{2}\right) \\
= & \left( \left\{ d_{1}d_{1}^{\prime },d_{1}^{\prime \prime }\right\}
+KD_{2}\right) +\partial _{1_{\ast }}(d_{1}^{\prime \prime }+KD_{1})\cdot
\left( \left\{ d_{1},d_{1}^{\prime }\right\} +KD_{2}\right) \\
= & \left\{ d_{1}d_{1}^{\prime }+KD_{1},d_{1}^{\prime \prime
}+KD_{1}\right\} +\partial _{1_{\ast }}(d_{1}^{\prime \prime }+KD_{1})\cdot
\left\{ d_{1}+KD_{1},d_{1}^{\prime }+KD_{1}\right\} \\
= & \left\{ \left( d_{1}+KD_{1}\right) \left( d_{1}^{\prime }+KD_{1}\right)
,d_{1}^{\prime \prime }+KD_{1}\right\} +\partial _{1_{\ast }}(d_{1}^{\prime
\prime }+KD_{1})\cdot \left\{ d_{1}+KD_{1},d_{1}^{\prime }+KD_{1}\right\} .%
\end{array}%
\end{equation*}

\textbf{PL4:}%
\begin{equation*}
\begin{array}{cl}
& \left\{ d_{1}+KD_{1},\partial _{2_{\ast }}\left( d_{2}+KD_{2}\right)
\right\} +\left\{ \partial _{2_{\ast }}\left( d_{2}+KD_{2}\right)
,d_{1}+KD_{1}\right\} \\
= & \left\{ d_{1}+KD_{1},\partial _{2}\left( d_{2}\right) +KD_{1}\right\}
+\left\{ \partial _{2}\left( d_{2}\right) +KD_{1},d_{1}+KD_{1}\right\} \\
= & \left( \left\{ d_{1},\partial _{2}\left( d_{2}\right) \right\}
+KD_{2}\right) +\left( \left\{ \partial _{2}\left( d_{2}\right)
,d_{1}\right\} +KD_{2}\right) \\
= & \left( \left\{ d_{1},\partial _{2}\left( d_{2}\right) \right\} +\left\{
\partial _{2}\left( d_{2}\right) ,d_{1}\right\} \right) +KD_{2} \\
= & \left( \partial _{1}(d_{1})\cdot d_{2}\right) +KD_{2} \\
= & \left( \partial _{1}(d_{1})+K\right) \cdot \left( d_{2}+KD_{2}\right) \\
= & \partial _{1_\ast }(d_{1}+KD_{1})\cdot \left( d_{2}+KD_{2}\right) .%
\end{array}%
\end{equation*}

\textbf{PL5:}%
\begin{equation*}
\begin{array}{ccl}
\left\{ d_{1}+KD_{1},d_{1}^{\prime }+KD_{1}\right\} \cdot r & = & \left\{
d_{1}+KD_{1},d_{1}^{\prime }+KD_{1}\right\} \cdot \left( s+K\right) \\
& = & \left( \left\{ d_{1},d_{1}^{\prime }\right\} +KD_{2}\right) \cdot
\left( s+K\right) \\
& = & \left( \left\{ d_{1},d_{1}^{\prime }\right\} \cdot s\right) +KD_{2} \\
& = & \left( \left\{ d_{1}\cdot s,d_{1}^{\prime }\right\} \right) +KD_{2} \\
& = & \left\{ \left( d_{1}\cdot s\right) +KD_{1},d_{1}^{\prime
}+KD_{1}\right\} \\
& = & \left\{ \left( d_{1}+KD_{1}\right) \cdot \left( s+K\right)
,d_{1}^{\prime }+KD_{1}\right\} \\
& = & \left\{ \left( d_{1}+KD_{1}\right) \cdot r,d_{1}^{\prime
}+KD_{1}\right\} .%
\end{array}%
\end{equation*}%
\begin{equation*}
\begin{array}{ccl}
\left\{ d_{1}+KD_{1},d_{1}^{\prime }+KD_{1}\right\} \cdot r & = & \left\{
d_{1}+KD_{1},d_{1}^{\prime }+KD_{1}\right\} \cdot \left( s+K\right) \\
& = & \left( \left\{ d_{1},d_{1}^{\prime }\right\} +KD_{2}\right) \cdot
\left( s+K\right) \\
& = & \left( \left\{ d_{1},d_{1}^{\prime }\right\} \cdot s\right) +KD_{2} \\
& = & \left( \left\{ d_{1},d_{1}^{\prime }\cdot s\right\} \right) +KD_{2} \\
& = & \left\{ d_{1}+KD_{1},\left( d_{1}^{\prime }\cdot s\right)
+KD_{1}\right\} \\
& = & \left\{ d_{1}+KD_{1},\left( d_{1}^{\prime }+KD_{1}\right) \cdot \left(
s+K\right) \right\} \\
& = & \left\{ d_{1}+KD_{1},\left( d_{1}^{\prime }+KD_{1}\right) \cdot
r\right\} .%
\end{array}%
\end{equation*}%
for all $d_{1}+KD_{1},d_{1}^{\prime }+KD_{1},d_{1}^{\prime \prime
}+KD_{1}\in D_{1}/KD_{1},d_{2}+KD_{2}\in D_{2}/KD_{2},r\in R$ and $s+K\in
S/K.$

Let us check that $\left( \phi ^{\prime \prime },\phi ^{\prime },\phi
\right) :\left\{ D_{2},D_{1},S,\partial _{2},\partial _{1}\right\}
\rightarrow \{ D_{2}/KD_{2},D_{1}/KD_{1},R,$ \\ $ \partial _{2_\ast },\partial
_{1_\ast }\} $ where $\phi ^{\prime \prime }\left( d_{2}\right)
=d_{2}+KD_{2}$ and $\phi ^{\prime \prime }\left( d_{1}\right) =d_{1}+KD_{1}$
is a morphism of $2$-crossed modules.

\begin{equation*}
\xymatrix {D_2 \ar[r]^{\partial _{2}} \ar[d]_{\phi ^{\prime \prime }} & D_1
\ar[r]^{\partial _{1}} \ar[d]_{\phi ^{\prime}} & S \ar[d]_{\phi} & \\
D_{2}/KD_{2} \ar[r]_-{\partial_{2_\ast} } & D_{1}/KD_{1}
\ar[r]_-{\partial_{1_\ast}} & R & }
\end{equation*}%
\begin{equation*}
\begin{array}{rcl}
\phi ^{\prime \prime }\left( s\cdot d_{2}\right) & = & \left( s\cdot
d_{2}\right) +KD_{2} \\
& = & \left( s+K\right) \cdot \left( d_{2}+KD_{2}\right) \\
& = & \phi \left( s\right) \cdot \phi ^{\prime \prime }\left( d_{2}\right)%
\end{array}%
\text{ and }\
\begin{array}{rcl}
\phi ^{\prime }\left( s\cdot d_{1}\right) & = & \left( s\cdot d_{1}\right)
+KD_{1} \\
& = & \left( s+K\right) \cdot \left( d_{1}+KD_{1}\right) \\
& = & \phi \left( s\right) \cdot \phi ^{\prime }\left( d_{1}\right)%
\end{array}%
\end{equation*}%
\begin{equation*}
\begin{array}{rcl}
\partial _{2_\ast }\left( \phi ^{\prime \prime }\left( d_{2}\right) \right) &
= & \partial _{2_\ast }\left( d_{2}+KD_{2}\right) \\
& = & \partial _{2}\left( d_{2}\right) +KD_{1} \\
& = & \phi ^{\prime }\left( \partial _{2}\left( d_{2}\right) \right)%
\end{array}%
\end{equation*}%
similarly $\partial _{1_\ast }\phi ^{\prime }=\phi \partial _{1}$ for all $d_{1}\in D_{1},d_{2}\in D_{2}$ and $s\in S.$
\begin{equation*}
\begin{array}{rcl}
\left( \left\{ -,-\right\} \left( \phi ^{\prime }\times \phi ^{\prime
}\right) \right) \left( d_{1},d_{1}^{\prime }\right) & = & \left\{
-,-\right\} \left( d_{1}+KD_{1},d_{1}^{\prime }+KD_{1}\right) \\
& = & \left\{ d_{1},d_{1}^{\prime }\right\} +KD_{2} \\
& = & \phi ^{\prime \prime }\left( \left\{ d_{1},d_{1}^{\prime }\right\}
\right) \\
& = & \left( \phi ^{\prime \prime }\left\{ -,-\right\} \right) \left(
d_{1},d_{1}^{\prime }\right) .%
\end{array}%
\end{equation*}
for all $d_{1},d_{1}^{\prime } \in D_{1}.$
\newline

\textbf{Acknowledgements}
\newline
This work was partially supported by T\"{U}B\.{I}TAK (The Scientific and Technological Research Council Of Turkey) for the second author.

U. Ege Arslan and G. Onarl\i \newline
Department of Mathematics and Computer Sciences \newline
Eski\c{s}ehir Osmangazi University \newline
26480 Eski\c{s}ehir/Turkey\newline
e-mails: \{uege, gonarli\} @ogu.edu.tr


\begin{thebibliography}{99}
\bibitem{aao} \textsc{U.E. Arslan}, \textsc{Z. Arvasi}, \textsc{G. Onarli},
\ {Induced Two-Crossed Modules, } http://arxiv.org/PS\_cache/arxiv/pdf/1107/1107.4291v2.pdf ,
(2011).

\bibitem{ae} \textsc{Z. Arvasi}, \textsc{U. Ege}, \ {Annihilators,
Multipliers and Crossed Modules,} Applied Categorical Structures,
11, (2003), 487-506.

\bibitem{ag} \textsc{U.E. Arslan}, \textsc{\"{O}. G\"{u}rmen}, Change of Base for Commutative Algebras,
Georgian Mathematical Journal, (to be published).

\bibitem{[pa]} \textsc{Z. Arvasi}, \textsc{T. Porter}, \ {Freeness Conditions of 2-Crossed Modules of Commuative Algebras,}
Applied Categorical Structures, 6, (1998), 455-471.

\bibitem{ap97} \textsc{Z. Arvasi}, \textsc{T. Porter}, \ {Higher Dimensional
Peiffer Elements in Simplicial Commutative Algebras,} Theory and
Applications of Categories, 3, (1997), 1-23.

\bibitem{a} \textsc{Z. Arvasi}, \ {Crossed Squares and 2-Crossed Modules of
Commutative Algebras,} Theory and Applications of Categories, 3,
(1997), 160-181.

\bibitem{[bh]} \textsc{R. Brown}, \textsc{P. J. Higgins}, \ {On the
Connection between the Second Relative Homotopy Groups of Some
Related Spaces,} Proc. London Math. Soc., 3, 36, (1978), 193-212.

\bibitem{[bhs]} \textsc{R. Brown}, \textsc{P. J. Higgins}, \textsc{R. Sivera}%
, \ {Nonabelian Algebraic Topology: filtered spaces, crossed
complexes, cubical higher homotopy groupoids}, European Mathematical
Society Tracts in Mathematics, Vol. 15, (2011).

\bibitem{[bs]} \textsc{R. Brown}, \textsc{R. Sivera}, \ {Algebraic colimit calculations in homotopy
theory using fibred and cofibred categories,} Theory and
Applications of Categories, 22, (2009), 222-251.

\bibitem{[bw1]} \textsc{R. Brown}, \textsc{C. D. Wensley},\ {Computing
Crossed Modules Induced by an Inclusion of a Normal Subgroup, with
Applications to Homotopy 2-types,} Theory Appl. Categ., 2, 1, 3-16,
(1996).

\bibitem{[bw2]} \textsc{R. Brown}, \textsc{C. D. Wensley}, \ {Computation
and Homotopical Applications of Induced Crossed Modules,} J.
Symbolic Comput., 35, (2003), 59-72.

\bibitem{cond} \textsc{D. Conduch\'{e}}, Modules Crois\'{e}s G\'{e}n\'{e}%
ralis\'{e}s de Longueur 2, J. Pure. Appl. Algebra, \ 34, (1984),
155-178.

\bibitem{G}  \textsc{M.Gerstenhaber}. On the Deformation of Rings and
Algebras, Ann. Math., 84, (1966).


\bibitem{GV} \textsc{A.R. Grandje\'{a}n}, \textsc{M. Vale},$\ $2-module en la
Cohomologia de Andr\'{e}-Quillen, Memorias de la Real
Academia de Ciencias, 22, (1986).

\bibitem{maclane} \textsc{S. Mac Lane}, Extension and Obstructions for Rings,
Illinois Journal of Mathematics, 121, (1958), 316-345.

\bibitem{L-S}  \textsc{S. Lichtenbaum} and \textsc{M. Schlessinger.} The
Cotangent Complex of a Morphism, Trans. American Society, 128,
(1967), 41-70.

\bibitem{lue} \textsc{A.S.T. Lue}, A Non Abelian Cohomology of
Associative Algebras, Quart. J.Math. Oxford Ser., 2, (1968),
159-180.

\bibitem{[p1]} \textsc{T. Porter}, Some Categorical Results in the
Category of Crossed Modules in Commutative Algebra, J. Algebra, 109,
(1978), 415-429.

\bibitem{[p2]} \textsc{T. Porter}, Homology of Commutative Algebras and
an Invariant of Simis and Vasconceles, J. Algebra, 99, (1986),
458-465.

\bibitem{PORTER} \textsc{T. Porter}, {The Crossed Menagerie: An Introduction
to Crossed Gadgetry and Cohomolgy in Algebra and Topology,}
http://ncatlab.org/timporter/files/menagerie10.pdf.

\bibitem{[s]} \textsc{N.M. Shammu}, \ {Algebraic and an Categorical
Structure of Category of Crossed Modules of Algebras,} Ph.D.
Thesis,U.C.N.W, (1992).

\bibitem{Vistoli} \textsc{A. Vistoli}, \ {Notes on Grothendieck Topologies, Fibered Categories and Descent Theory,} http://homepage.sns.it/vistoli/descent.pdf, (2008).

\bibitem{[w1]} \textsc{J.H.C. Whitehead}, \ Combinatorial Homotopy I,%
 Bull. Amer. Math. Soc., 55, (1949),
231-245.

\bibitem{[w2]} \textsc{J.H.C. Whitehead}, \ {Combinatorial Homotopy II,}
Bull. Amer. Math. Soc., 55, (1949), 453-456.
\end{thebibliography}
\end{document}